\documentclass[11pt]{amsart}
\usepackage{amsmath,amsthm,amscd,amssymb, color}
\usepackage{latexsym}
\usepackage[colorlinks, citecolor = blue]{hyperref}
\usepackage[alphabetic]{amsrefs}
\usepackage[capitalize, nameinlink]{cleveref}
\usepackage{enumerate}
\usepackage[margin=1.35in]{geometry}

\newcommand{\R}{\mathbb R}

\newcommand{\Q}{\mathbb Q}
\newcommand{\C}{\mathbb C}
\newcommand{\Z}{\mathbb Z}

\newcommand{\bfk}{\mathbf k}
\renewcommand{\phi}{\varphi}
\newcommand{\eps}{\varepsilon}
\newcommand{\frakp}{\mathfrak p}
\newcommand{\frakP}{\mathfrak P}
\newcommand{\fraku}{\mathfrak u}

\newcommand{\fraka}{\mathfrak a}
\newcommand{\frakb}{\mathfrak b}
\newcommand{\frakc}{\mathfrak c}
\newcommand{\frakd}{\mathfrak d}
\newcommand{\frakn}{\mathfrak n}
\newcommand{\frakN}{\mathfrak N}
\newcommand{\frakM}{\mathfrak M}
\newcommand{\frakD}{\mathfrak D}
\newcommand{\frako}{\mathfrak o}
\newcommand{\calO}{\mathcal O}

\newcommand{\calI}{\mathcal I}
\newcommand{\calU}{\mathcal U}
\newcommand{\A}{\mathbb A}
\newcommand{\bs}{\backslash}
\newcommand{\bmx}{\left( \begin{matrix}}
\newcommand{\emx}{\end{matrix} \right)}
\newcommand{\Cl}{\mathrm{Cl}}

\newcommand{\JL}{\mathrm{JL}}
\newcommand{\St}{\mathrm{St}}
\newcommand{\sgn}{\mathrm{sgn}}
\newcommand{\new}{\mathrm{new}}
\newcommand{\old}{\mathrm{old}}
\renewcommand{\mod}{\, \, \mathrm{mod} \, \,}

\newcommand{\one}{\mathbf{1}}
\newcommand{\two}{\mathbf{2}}
\newcommand{\zero}{\mathbf{0}}

\DeclareMathOperator{\GL}{GL}

\DeclareMathOperator{\tr}{tr} 
\DeclareMathOperator{\Ind}{Ind} 
 
\DeclareMathOperator{\Sym}{Sym} 

\newtheorem{lemma}{Lemma}
\numberwithin{lemma}{section}
\newtheorem{prop}[lemma]{Proposition}
\newtheorem{thm}[lemma]{Theorem}
\newtheorem{cor}[lemma]{Corollary}

\theoremstyle{remark}
\newtheorem{rem}{Remark}
\numberwithin{rem}{section}

\numberwithin{equation}{section}

\begin{document}

\title{The basis problem revisited}
\author{Kimball Martin}
\address{Department of Mathematics, University of Oklahoma, Norman, OK 73019}
\email{kimball.martin@ou.edu}


\date{\today}

\begin{abstract}
Eichler investigated when there is a basis of a space of modular forms
consisting of theta series attached to quaternion algebras, and treated
squarefree level.  Hijikata, Pizer and Shemanske
completed the solution to Eichler's basis problem for
elliptic modular forms of arbitrary level by tour-de-force trace calculations.
We revisit the basis problem using the representation-theoretic perspective
of the Jacquet--Langlands correspondence.  

Our results include:
(i) a simpler proof of the solution to the basis problem for elliptic modular forms, which also
allows for more flexibility in the choice of quaternion algebra;
(ii) a solution to the basis problem for Hilbert modular forms; (iii) a theory of
(local and global) new and old forms for quaternion algebras; and (iv)
an explicit version of the Jacquet--Langlands correspondence at the level of
modular forms, which is a  refinement of the Hijikata--Pizer--Shemanske
solution to the basis problem.  Both (i) and (ii) have practical applications to computing elliptic and
Hilbert modular forms.  Moreover, (iii) and (iv) are desired for arithmetic applications---to
illustrate, we give a simple application to Eisenstein congruences in level $p^3$.
\end{abstract}

\maketitle
\tableofcontents

%
%

\section{Introduction}

Let $S_k(N)$ denote the space of elliptic cusp
forms of weight $k$ on $\Gamma_0(N)$ with trivial character.  
Denote by $B$ a definite quaternion algebra over $\Q$ (for the moment),
which is characterized by its discriminant $D_B$, a positive 
squarefree product of an odd number of primes (the finite places of ramification of
$B$).  Denote by $\calO$ an order in $B$.

\subsection{The basis problem: a classical history}

In the case $D_B = p$ is prime and $\calO$ is a maximal order of $B$, 
Hecke (1940) conjectured that the (differences of) theta series associated to 
a set of one-sided $\calO$-ideal class representatives yield a basis for $S_2(p)$.  While this conjecture was not quite
correct (the number of theta series here equals $1+\dim S_2(p)$,
but they often are linearly dependent), 
Eichler \cite{eichler:1955}
showed that there is a basis for $S_2(p)$ coming from a larger
collection of theta series associated
to $\calO$.  These theta series come from forming
Fourier series $\sum a_{ij}(n) q^n$, where $a_{ij}(n)$ is the $(i,j)$-th entry
of the $n$-th Brandt matrix $A_n$ associated to $\calO$.  Eichler's proof
relies on a comparison of traces of Brandt matrices with traces of Hecke 
operators on $S_2(p)$ via explicit computation.

More generally, Eichler considered the question of whether the newspace 
$S^\new_k(N)$ of $S_k(N)$ has a basis consisting of theta series attached to an 
order $\calO$ in a definite quaternion algebra, which is called the (Eichler) basis
problem.  Eichler \cite{eichler} extended his approach from \cite{eichler:1955}
to show that this question has a positive answer when 
$N > 1$ is squarefree using what are now known as Eichler orders $\calO$ (intersections of at most 2 maximal orders).
Note one cannot construct elements of $S_k(1)$ via quaternionic theta
series, but there are well-known ways for constructing a basis for $S_k(1)$,
e.g., with the Eisenstein series $E_4$ and $E_6$.  Hence \cite{eichler}
together with Atkin--Lehner theory provides a way of constructing bases for 
$S_k(N)$ for any squarefree $N$.

Using more general Eichler orders, Hijikata and Saito \cite{hijikata-saito} 
extended Eichler's results to levels of the form $N=pM$.  Here $M$ denotes 
a positive integer prime to $p$.
This was further generalized by Pizer in \cite{pizer:algorithm} and \cite{pizer:p2} 
where he treated levels of the form $N=p^{2r+1}M$ and $N=p^2 M$.  However,
in the case $N=p^2M$, Pizer needs to assume $p$ is odd, and now it is only true
that $S^\new_k(N)$ is generated by quaternionic theta series and twists of
newforms of levels $M$ and $pM$ (with character).
Pizer followed a similar approach to Eichler's, but needed to work with
non-Eichler orders.  (An Eichler order is maximal at primes
dividing $D_B=p$, so there are no Eichler orders of level $p^r M$ for $r > 1$.)

The only levels that remain are those of the form $N=(N')^2$ where no odd prime
sharply divides $N'$.  To treat these, Hijikata, Pizer and Shemanske 
\cite{HPS:crelle} introduced a generalization of Eichler's and Pizer's orders, called
special orders.  Then in a tour-de-force calculation \cite{HPS:mams}, they
solved the basis problem using special orders of level $N=p^r M$ in a
definite quaternion algebra of discriminant $p$.  The basic argument follows
Eichler's original approach, but the necessary calculations with special orders
are much more complicated (especially for $N=2^{2r}M$), and again
one needs to consider twists of forms of smaller level (see \cref{cor:intro} below).

\subsection{Connection with the Jacquet--Langlands correspondence}
The solution to the basis problem may be viewed as a classical interpretation
of the representation-theoretic Jacquet--Langlands correspondence,
which is an injective correspondence from 
(irreducible, infinite-dimensional) automorphic 
representations $\pi$ of $B^\times$ to 
(irreducible, infinite-dimensional) cuspidal automorphic representations $\pi'$
of $\GL(2)$.  It is well known how to view $S_k(N)$ as a sum of certain
invariant subspaces of appropriate $\pi'$.  One can also define
a space of \emph{quaternionic modular (cusp) forms} $S_k(\calO)$, which can be
viewed as a sum of certain invariant subspaces of appropriate $\pi$.
Here the analogue of level is played by an order $\calO$ in $B$, and
$S_k(\calO)$ may be viewed as certain (vector-valued if $k > 0$) functions
on the set of right $\calO$-ideal classes.  Then the Jacquet--Langlands
correspondence may be viewed ``classically'' as a non-canonical linear map
\begin{equation} \label{eq:cJL}
 \JL : S_k(\calO) \to S_{k+2}(N) \qquad (\text{where } k \ge 0),
\end{equation}
which preserves the action of unramified Hecke operators.  This applies
to any order $\calO$ in any $B$, and some $N$ 
depending on $\calO$.  This is related to the basis problem by
associating weight $k+2$ theta series to  $S_k(\calO)$.

However, the proof of the Jacquet--Langlands correspondence does
not answer any of the following questions: (1) given $\calO$, what can we
take for $N$?  (2) what is the kernel of this map?  and (3)
what is the image of this map?  Note that while the map JL in \eqref{eq:cJL} is 
not canonical, if one defines this map so the kernel has minimum possible
dimension, then the kernel and image are canonical up to isomorphism as
Hecke modules. In fact, one can define
JL so that the image is spanned by theta series, which specifies the
image exactly, though then it is not clear if the kernel necessarily has
minimum possible dimension on the old space.

The first question is essentially asking: ($1'$)
how does ``level'' behave under the local Jacquet--Langlands
correspondence?  An answer to this will also let us describe the image of this
map, using the representation-theoretic characterization of the image of
the Jacquet--Langlands correspondence.  The solution to the basis problem in
\cite{HPS:mams} says that, if $D_B = p$ and $\calO$ is a special
order of level $p^r M$ (of a certain type if $r$ is odd), then one can take 
$N=p^r M$ and describe the portion of the image lying in
$S_{k+2}^\new(N)$.  However, \cite{HPS:mams} does not tell us everything we
want to know: it does not completely answer (2)
or (3), or the local analogue ($1'$), or handle general $B$.

Conversely, understanding the ``classical'' Jacquet--Langlands correspondence
\eqref{eq:cJL}
can be applied to the basis problem, and this may be approached via
representation theory.  That is the goal of the present work.  This applies to
arbitrary totally real base fields $F$, and we obtain both more information about the
classical Jacquet--Langlands correspondence and the basis problem when $F=\Q$,
as well a solution to the basis problem for Hilbert modular forms.

To our knowledge, this idea was first realized by Shimizu in \cite{shimizu}, 
albeit in a restricted setting.  In that work, he gave a theta series proof of the
Jacquet--Langlands correspondence (under some conditions), and applied
this to the basis problem over totally real fields $F$ for squarefree level 
$\frakN$ under a parity condition on $[F:\Q]$ and $\frakN$ (and weight $> 2$
at each infinite place).  Shimizu's condition corresponds to using a maximal order
$\calO$ in a definite quaternion algebra $B/F$ of discriminant $\frakN$.
The representation-theoretic approach makes it easy to get a solution
to the basis problem more generally when $[F:\Q]$ is even or
when $\frakN$ is divisible by some prime ideal $\frakp$ such that $\frakp^2 \nmid
\frakN$.  (Shimizu's condition on the weight is not essential for our purposes.  See also
\cite[Theorem 3]{waldspurger} for a more classical treatment when 
$[F:\Q]$ is even.)  
When $F=\Q$,
this is the case treated by Hijikata and Saito \cite{hijikata-saito}.  This
extension of Shimizu's application to the basis problem simply comes from working with Eichler orders 
instead of maximal orders, and using a quaternion algebra ramified at at most one finite place.  (Shimizu's theta series
are not presented in a classically explicit way like Eichler's---cf.\ \cite{gebhardt}---but ours below will be.)

\subsection{Summary of results}

Let $B$ be a definite quaternion algebra of discriminant $\frakD$ 
over a totally real number field $F$.  Let $\calO$ be a special order of level
$\frakN$.  This means, for $\frakp$ a finite prime of $F$, 
$\calO_\frakp$ is an Eichler order for
$\frakp \nmid \frakD$ and $\calO_\frakp$ contains the ring of integers
of a quadratic extension $E_\frakp/F_\frakp$ for $\frakp | \frakD$, and that
the product of the levels of the local orders is $\frakN$.  (We use the convention
that the maximal order $\calO_\frakp$ for $\frakp | \frakD$ has level $\frakp$,
rather than level $\frako_F$.)
Necessarily, $\frakD | \frakN$.  
We further assume $E_\frakp/F_\frakp$ is unramified when $v_\frakp(\frakN)$ is odd.
Such special orders exist for all multiples $\frakN$ of $\frakD$.

Let $d=[F : \Q]$ and $\bfk = (k_1, \ldots, k_d)$ with each $k_i \in \Z_{\ge 0}$.
Let $M_\bfk(\calO)$ (resp.\ $S_\bfk(\calO)$) be the space of quaternionic
modular (resp.\ cusp) forms of weight $\bfk$, level $\calO$, and trivial character.
(These spaces will be denoted $M_\bfk(\calO, 1, 1)$ and $S_\bfk(\calO, 1, 1)$ in
\cref{sec:qmfs}, where we treat quaternionic modular forms with character.)
In fact $S_\bfk(\calO) = M_\bfk(\calO)$ unless $\bfk = \zero = (0, \ldots, 0)$.
In \cref{sec:QMFs},
we define Hecke operators $T_\frakn$ on $M_\bfk(\calO)$, and associated
Brandt matrices which realize their action with respect to a suitable basis.

We first develop a theory of quaternionic newforms and oldforms along the lines
of Casselman's approach \cite{casselman} to Atkin--Lehner theory.  
This relies on a decomposition of $M_\bfk(\calO)$ as a direct sum of invariant
subspaces $\pi_f^K$, where $\pi = \otimes \pi_v$ runs over automorphic representations
of $B^\times$ of ``weight $\bfk$'' and 
$K = \hat \calO^\times = \prod_\frakp \calO_\frakp^\times$.  

\subsubsection{Local results}
In \cref{sec:local-new}, we study the dimensions of $\pi_\frakp^{K_\frakp}$
for $\frakp | \frakD$ (and to treat general characters, more generally the restriction of
$\pi_\frakp$ to $\calO_\frakp^\times$).  Since $\calO_\frakp^\times \supset
\frako_{E_\frakp}^\times$, this is closely
related to the restriction problem for $(B_\frakp^\times, E_\frakp^\times)$,
which was solved by Tunnell  \cite{tunnell:eps} and
Saito \cite{saito} in terms of epsilon factors.
In particular, $\dim \pi_\frakp^{K_\frakp} \le e(E_\frakp/F_\frakp)$,
the ramification index.  We get an essentially complete answer when
$E_\frakp/F_\frakp$ is unramified, but only a partial answer when $E_\frakp/F_\frakp$ is 
ramified.  The former case gives a simple formula for epsilon factors
in certain situations (\cref{rem:gross}).
In the latter case, the obstruction to a complete answer
is that one cannot prove such a simple formula for relevant epsilon factors.

In contrast to local newform theory for $\GL(2)$, there are
two new phenomena here.  First, for fixed $\pi_\frakp$, $\dim \pi_\frakp^{K_\frakp}$ essentially does not increase upon raising the level of $K_\frakp$ beyond the
``conductor'' of $\pi_\frakp$.  This is perhaps not surprising as these representations 
are finite dimensional.  Second, at the lowest level $K_\frakp$ for which 
$\pi_\frakp^{K_\frakp}$ is nonzero, $\pi^{K_\frakp}$ 
may be 2-dimensional when $E_\frakp/F_\frakp$
is ramified (and it is when $\pi_\frakp$ corresponds to a minimal supercuspidal of $\GL(2)$ of even conductor).  

A consequence of our study of local representations is 
a description of how ``level''  behaves along the Jacquet--Langlands 
correspondence, i.e., an answer to question ($1'$), and thus also (1), above.  
In general, for functorial transfers, one often knows depth is
preserved, but the behavior of level is more mysterious.  To our knowledge, 
this is the first complete description of the behavior of level for local functorial
transfer between two groups $G$ and $G'$ where one group is not quasi-split.
Thus our local results may be viewed as a baby case of this general problem.

For general groups $G$, there is no canonical way to define level (conductor),
particularly for non-quasi-split groups, so the above choice of $K_\frakp$ 
may suggest reasonable analogues to consider in rather general situations.
We venture that a relevant property of the compact open subgroup
$K_\frakp = \calO_\frakp^\times$ here is that it contains a maximal compact 
subgroup (here $\frako_{E_\frakp}^\times$)
of a subgroup (here $E_\frakp^\times$) of $G=B_\frakp^\times$ which
possesses the multiplicity one property for restriction of representations.

\subsubsection{Global results}
Using our local results, we give an Atkin--Lehner type 
global decomposition of the space of 
quaternionic modular forms into new spaces of smaller levels
(see \cref{cor:quat-AL-decomp2} and \eqref{eq:AL-decomp2}
for cusp forms, and \eqref{eq:AL-eis-decomp} for Eisenstein series).

With this decomposition of $S_\bfk(\calO)$, and our study of level in the local
Jacquet--Langlands correspondence \eqref{eq:cJL}, we give a description of
both the kernel and the image (questions (2) and (3) above) of the classical
Jacquet--Langlands correspondence.  However, due to our incomplete 
local results
when $E_\frakp/F_\frakp$ is ramified, we do not get a complete description
in all situations.  Things are simpler when we restrict to the
new space, where the hypothesis we need to get a complete description is
that any dyadic $\frakp | \frakD$ satisfies $v_\frakp(\frakN)$ is odd, or 2, or sufficiently
large (see \cref{JL-thm-new} and \cref{rem:2-powers}).
For the full space $S_\bfk(\calO)$, we give a complete description of $\JL$
when any prime $\frakp | \frakD$ satisfies $v_\frakp(\frakN)$ is odd or 2
(see \cref{cor:cusp-JL}).  This condition means that every $\pi$
appears in $S_\bfk(\calO)$ will be locally minimal or 1-dimensional for
each $\frakp | \frakD$.

For simplicity, we only properly state some of our results now
(see \cref{prop:JL-prim} and \cref{cor:cusp-JL} for more general statements).  Let $\two = (2, \ldots, 2)$,
and write $\frakN = \frakN' \frakM$, where $\frakM$ is the part of $\frakN$
coprime to $\frakD$.  For $\frakd | \frakN$, let $S_\bfk^{\frakd \textrm{-} \new}(\frakN)$
be the subspace of $S_\bfk(\frakN)$ consisting of forms which are $\frakp$-new
for all $\frakp | \frakd$.  We also say a form is $\frakp$-primitive if there is no twist
which lowers the level at $\frakp$.

\begin{thm} \label{intro-thm1} There is a Hecke-module homomorphism
$\JL:S_\bfk(\calO) \to S_{\bfk+\two}(\frakN)$ such that

\begin{enumerate}[(i)]
\item any newform $f \in S_{\bfk+\two}(\frakN)$ which is $\frakp$-primitive for
$\frakp | \frakD$ is contained in the image;

\item if $v_\frakp(\frakN)$ is odd for all $\frakp | \frakD$, then $\JL$ is injective
and yields an isomorphism 
\[ S_\bfk(\calO) \simeq \bigoplus_{\frakd} S_{\bfk + \two}^{\frakd \textrm{-} \new}(\frakd \frakM), \]
where $\frakd$ runs over all divisiors of $\frakN'$ such that $v_\frakp(\frakd)$
is odd for all $\frakp | \frakD$.
\end{enumerate}
\end{thm}

The (notational) complication in the description when 
$v_\frakp(\frakN) = 2$ for some $\frakp | \frakD$ is that the kernel is generally nonzero,
and its description depends not just on the conductor of $\pi_\frakp$, but whether
that representation is 1-dimensional or not.  For instance, in the simple case that
$F=\Q$, $D_B = p$, and $\calO$ has level $p^2$, then we have
\[ S_k(\calO) \simeq S^\new_{k+2}(p) \oplus S^{\text{new-sp}}_{k+2}(p^2)
\oplus 2 S^{\text{new-sc}}_{k+2}(p^2), \]
where the second (resp.\ third) space on the right denotes the span of newforms whose 
local representation is special (resp.\ supercuspidal) at $p$.  These first two spaces
on the right correspond to the associated local representation $\pi_p$ of 
$B_p^\times$ being 1-dimensional.

The issue when $v_\frakp(\frakN)$ is some higher even power
for a $\frakp | \frakD$ is that we cannot say what forms $f$ whose $\frakp$-power level
is odd appear in the image of the Jacquet--Langlands correspondence.  If such
a $\frakp$ is dyadic, we also cannot always say what non-minimal 
$\frakp$-new
forms appear.  Modulo these issues,
we can describe $S_\bfk(\calO)$ completely as a Hecke module in terms of
Hilbert modular forms for an arbitrary $\frakN$, and consequently the map JL.

We note that, when restricted to new spaces, the image of JL is always canonical.
However, even restricted to new spaces, the kernel of JL is not canonical when it
is nonzero.  Note the kernel on the new space is in general nonzero in the above 
level $p^2$ example, or more generally when primes occur to even
powers in $\frakN'$.  See \cref{rem:theta-ker} for the dimension of the kernel restricted
to the new space.

Part (ii) of the above theorem specializes to the known
extension of \cite{shimizu} when $v_\frakp(\frakN) = 1$ for all $\frakp | \frakD$.
If further $F=\Q$, $k=0$ and $\calO$ is maximal, 
a different proof follows from \cite{ponomarev}.  (The latter
paper works with theta series, but its result is equivalent to the above type of
statement via \cref{qmf-to-theta} and \cref{theta-thm2}.)
 
Finally, in \cref{sec:theta}, we explain how to associate spaces of theta series
to spaces of quaternionic modular forms via Brandt matrices.  This
gives a realization of the map JL in terms of classical theta series.
Our treatment of theta series is also
rather representation theoretic, as opposed to Eichler's 
treatments in \cite{eichler} and \cite{eichler:1977}.  
A key ingredient is a comparison
of two different definitions of local ramified Hecke operators (both for $\frakp | \frakD$
and $\frakp \nmid \frakD$).
This leads to our solution to the basis problem.
Again, for simplicity of exposition, we will not state our results here 
in full generality or precision.

First, we have the following ``weak solution'' to the basis problem (see
\cref{cor:basis} and the subsequent discussion), generalizing the $F=\Q$
case of \cite{HPS:mams}.

\begin{thm} \label{thm:intro}
The space $S_{\bfk+\two}^\new(\frakN)$ is linearly generated
by theta series associated to $S^\new_\bfk(\calO)$ together with twists 
of Hilbert modular forms (with character) of lower level.
\end{thm}

By suitably varying our quaternion algebra $B$ and order $\calO$, 
we get the following solution to the basis problem (see the discussion after
\cref{cor:basis}).

\begin{cor} \label{cor:intro}
Any space $S_\bfk^\new(\frakN)$ is linearly generated
by twists of suitable quaternionic theta series, unless $[F:\Q]$ is odd
and $\frakN$ is a perfect square, in which case the space is generated by
twists of quaternionic theta series together with twists of forms of level $1$
and nebentypus conductor $\frakN^{1/2}$.
\end{cor}

We note that, like the solution in \cite{HPS:mams}, this requires
using quaternionic theta series ``with character'' for lower levels, 
so we must work with quaternionic modular forms and Brandt matrices
with character.
At the end of \cref{sec:theta}, we also explain what one can say about the basis problem for 
Hilbert modular forms with character.

Finally, we describe two applications.
First, even when $F=\Q$, \cref{intro-thm1}, and
the corresponding application to the basis problem, is new.
Namely we do not restrict to quaternion algebras of prime discriminant.
Computationally, say to use Brandt matrices to compute
spaces of modular forms as in \cite{pizer:algorithm} or \cite{DV}, 
this is desirable.  E.g., if we want to compute the newforms of level
$N = (pqr)^3$, where $p$, $q$, $r$ are distinct primes, 
we can use an order of level $N$ in the quaternion algebra of discriminant $pqr$.
This will only pick up newforms of level $p^e q^f r^h$, where $e, f, h \in \{ 1, 3 \}$,
and each of them only once.  On the other hand if we were to use a quaternion
algebra ramified only at $p$, we will pick up many more old forms at $q$ and $r$,
complicating the calculations.

More generally, if one wants to focus on studying newforms, it is often difficult
to isolate them analytically on $\GL(2)$.  However by working with 
quaternion algebras which are ``as ramified as possible'' one can eliminate 
most old forms.

Second, many arithmetic properties of modular forms such as congruences
and $L$-values are studied by using (definite and indefinite) quaternion
algebras and the Jacquet--Langlands correspondence.  For some problems,
understanding the representation-theoretic Jacquet--Langlands correspondence
suffices, but for others one wants to understand it at the level of modular forms.
One example of the latter is our construction of Eisenstein congruences
via quaternionic modular forms \cite{me:cong}.  One specific result
there is that, for $p$ odd, there is always a mod $p$ Eisenstein congruence
in $S_2(p^3)$, but it was not clear to us at the time how to prove there is such a 
congruence in $S_2^\new(p^3)$. 
In \cref{sec:cong}, we show how to deduce this from \cref{intro-thm1}
and \cite{me:cong}.  We confess that our original motivation
arose, not from an interest the basis problem, but from trying to understand 
spaces of quaternionic modular forms for applications to congruences and $L$-values.

\medskip
\noindent
{\bf Acknowledgements.}  I would like to thank Tom Shemanske for discussions
about his work with Hijikata and Pizer.  I 
also thank Lassina Demb\'el\'e and John Voight for kind feedback.  I am grateful
to the anonymous referees, especially one particularly diligent referee,
for very detailed and insightful comments and corrections. 
This work was supported by grants from the 
Simons Foundation/SFARI (240605 and 512927, KM).

%
%

\section{Local preliminaries}

In this section and the next, we keep the following notation.

Let $F$ be a finite extension of $\Q_p$ with residue degree $q$.
For a finite extension $E/F$ of fields, we let $\frako_E$ denote the ring of integers,
$\frakp_E$ the prime ideal, $\fraku_E^0 = \frako_E^\times$ the unit group,
$\fraku_E^n = 1 + \frakp_E^n$
the $n$-th higher unit group ($n \ge 1$), $\varpi_E$ a uniformizer, $v_E$
the (exponential) valuation normalized so $v_E(\varpi_E) = 1$,
$N=N_{E/F}$ the norm map, and $\tr=\tr_{E/F}$ the trace map.
For a character $\chi$ of $E^\times$ (or even just of
$\frako_E^\times$), let $c(\chi)$ denote its
conductor, i.e., the minimal $n \ge 0$ such that $\fraku_E^n \subset \ker 
\chi$.
When $E=F$, we usually omit the subscript $F$.

Denote by $B$ the unique quaternion division algebra over $F$.
Let $N=N_{B/F}$ and $\tr = \tr_{B/F}$ denote the reduced norm
and trace maps.  The valuation $v_F$ induces a valuation
$v_B$ on $B$ such that $v_B(\alpha) = v_F(N(\alpha))$ for $\alpha \in B$.
Then $\calO_B = \{ \alpha \in B : v_B(\alpha) \ge 0 \}$ is the unique
maximal order of $B$ with (2-sided) prime ideal $\frakP = \{ \alpha \in B :
v_B(\alpha) \ge 1 \} = \varpi_B \calO_B$, where $\varpi_B$ denotes
a uniformizer of $B$, i.e., an element of $B$ with valuation 1.
We also define the higher unit groups $\calU^0 = \calO_B^\times$
and, for $n \ge 1$, $\calU^n = 1 + \frakP^n$.  
Note $\calU^n$ is a normal subgroup of $\calO_B^\times$
as conjugation stabilizes $\frakP^n$, and the collection $\{ \calU^n \}$
forms a neighborhood basis of compact open sets for the identity in $B^\times$.

We will need some facts about quotients of unit groups.  First note that
$\calO_B/\frakP$ is a finite division ring, and thus a field.  It has order $q^2$.
Hence $\calU^0/\calU^1$ is commutative and cyclic of order $q^2-1$.
For $n \ge 1$, any successive quotient $\calU^n/\calU^{n+1}$ has order $q^2$,
and is also abelian.  In fact, since
\begin{equation} \label{eq:abel-quo}
 (1+\varpi_B^n x)(1+\varpi_B^n y) \in 1+\varpi_B^n(x+y) + \frakP^{2n},
\quad x, y \in \calO_B,
\end{equation}
we see that $\calU^n/\calU^{2n}$ is abelian of order $q^{2n}$ for each $n$.

Suppose $E/F$ is an extension which embeds in $B$.  This means $E=F$
or $E$ is any quadratic extension of $F$.  If $e=e(E/F)$ denotes the ramification
index of $E/F$, then we have $\frakP^n \cap E = \frakp_E^{\lceil ne/2 \rceil}$,
where $\lceil x \rceil$ denotes the least integer (ceiling) function.
Thus $\calU^n \cap E = \fraku_E^{\lceil ne/2 \rceil}$.

If $x \in \frakP^n$, then $N(1+x) \in 1 + \tr x + \frakP^{2n} \subset 1 + \frakp^{\lceil n/2 \rceil}$.
If $E/F$ is the unramified quadratic extension, then
$\frakp_E^{\lceil n/2 \rceil} \subset \frakP^n$ and since $N_{E/F}$ is surjective on
higher unit groups, we see that in fact 
$N(\calU^n) = \fraku^{\lceil n/2 \rceil}$.

If $E/F$ is a quadratic extension, let $t=t(E/F) = v_E(\bar x - x) - 1$
where $x \in E$ is such that $\frako_E = \frako_F + x \frako_F$.
Then $t=-1$ if and only if $E/F$ is unramified, and $t=0$ if and only if
$E/F$ is ramified with odd residual characteristic.  If $E/F$ is unramified,
we have $N(\fraku_E^n) = \fraku^n$ for each $n \ge 0$.  If $E/F$ is ramified,
then for each $n \ge 0$, we have 
\begin{equation} \label{eq:norm-ram}
N(\fraku_E^{t+2n+1}) = N(\fraku_E^{t+2n+2}) = \fraku^{t+n+1}.
\end{equation}
In particular, if $E/F$ is ramified with odd residual characteristic, we have
$N(\fraku_E^{n}) = \fraku^{\lceil n/2 \rceil}$ for all $n \ge 1$.
See \cite{HPS:crelle} for more details.

\subsection{Special orders}

Let $E/F$ denote a quadratic extension in $B$.  In \cite{HPS:crelle}, 
Hijikata, Pizer and Shemanske considered the local orders
\begin{equation}
\calO_r(E) = \frako_E + \frakP^{r-1} \qquad (r \ge 1),
\end{equation}
which they termed \emph{special}.  The \emph{level} of a special
order $\calO$ is $\frakp^r$ if $r$ is minimal such that $\calO \simeq \calO_r(E)$
for some $E$.
Observe that, for any $E$, $\calO_1(E)=\calO_B$ is the maximal
order.  Also note that $\calO_r(E)^\times = \frako_E^\times \calU^{r-1}$.

We recall a few facts about special orders from \emph{op.\ cit.}

First, an order $\calO$ of $B$ is of the form $\calO_r(E)$ for some
$r$ if and only if $\frako_E$ embeds in $\calO$.  

If $E/F$ is unramified, then $\calO_{2m-1}(E) = \calO_{2m}(E) \ne \calO_{2m+1}(E)$ for 
any $m \ge 1$.
In particular, $\frako_E^\times \calU^{2m} = \frako_E^\times \calU^{2m+1}$,
which in fact holds for $m \ge 0$.  For $E/F$ unramified, the level of
$\calO_r(E)$ always has odd valuation.

If $E/F$ is ramified, then the orders $\calO_r(E)$ are distinct for all $r \ge 1$.
Note that $\frako_E = \frako_F + \frakp_E$ as the index of the latter in the former
is $[\frako_E:\frakp_E]/[\frako_F : \frakp_F] = q/q =1$.
So if $K/F$ is another ramified quadratic extension in $B$,
then $\frako_E \subset
\calO_2(K)$ as $\frako_F \subset \frako_K$ and $\frakp_E \subset \frakP$.
Hence $\calO_2(E) \subset \calO_2(K)$ and conversely, i.e., 
$\calO_r(E) = \calO_r(K)$ for $r = 1, 2$.  On the other hand, when $q$ is odd,
$\frako_E$ does not embed in $\calO_3(K)$, so $\calO_r(E)$ and $\calO_r(K)$
are non-isomorphic for all $r \ge 3$ when $K \not \simeq E$.  
(See \cite[Lemma 3.9]{HPS:crelle} for the case of $q$ even.)

\subsection{Local representations} \label{sec:loc-rep}

Next we recall some facts about the representation theory of $B^\times$,
most of which can be found in, or easily derived from, some combination of
\cite{BH}, \cite{tunnell:llc} and \cite{tunnell:eps}.  In any case, we at least
indicate the proof for facts not explicitly stated in \cite{BH} (e.g., facts
about conductors and dimension formulas).

Let $\pi$ be a smooth irreducible complex representation of $B^\times$.
Let $\omega_\pi$ be the central character of $\pi$.  
For a character $\mu$ of $F^\times$, denote by $\pi \otimes \mu$ the twist
$(\mu \circ N_{B/F}) \cdot \pi$, which has central character $\omega_\pi \mu^2$.
So we can choose a twist $\pi \otimes \mu$ which is trivial on $\langle 
\varpi \rangle$.  Hence, up to twisting,
we can view $\pi$ as a representation of the compact quotient $B^\times/
\langle \varpi \rangle$, and thus $\pi$ is finite dimensional.

By smoothness there is some $n$ such that
$\pi$ restricted to $\calU^n$ acts trivially, i.e., $\calU^n \subset \ker \pi$.  
The minimal integer $\ell \ge 0$
such that $\calU^{\ell + 1} \subset \ker \pi$ is called the (unnormalized) 
level of $\pi$, and is denoted by $\ell = \ell(\pi)$.

Denote by $\pi' = \JL(\pi)$ the Jacquet--Langlands transfer of $\pi$ to $\GL_2(F)$,
which is a discrete series representation of $\GL_2(F)$ with the same
central character as that of $\pi$.  Set $c(\pi) = c(\pi')$, where 
$c(\pi')$ is the conductor of $\pi'$.  Alternatively, we can
define $c(\pi)$ to be the minimal $n \ge 1$ such that $\pi$ is trivial on $\calU^{n-1}$,
and thus $\ell(\pi) = \max \{ c(\pi) - 2, 0 \}$.  
That these definitions are the same is easy to see when
$\dim \pi = 1$ (see below).  For general $\pi$, it
follows from the fact that $\ell(\pi') = \frac 12 \ell(\pi)$, using the normalization of level
for $\GL_2(F)$ as in \cite{BH} (see Section 56.1), and comparing $\ell(\pi')$
with $c(\pi')$ via $\eps$-factor relations.  
Note that
$\omega_\pi$ must be trivial when restricted to 
$\calU^{c(\pi) - 1} \cap F$,
so $c(\omega_\pi) \le \lfloor \frac{c(\pi)}2 \rfloor$.

We say $\pi$ is minimal if $c(\pi) \le c(\pi \otimes \mu)$ for all $\mu$; equivalently, if $\ell(\pi) \le \ell(\pi \otimes \mu)$ for all $\mu$
and $c(\pi) = 1$ if $\dim(\pi) = 1$.
If $\pi$ is minimal, then
$c(\pi \otimes \mu) = \max \{ c(\pi), 2c(\mu) \}$ \cite[Proposition 3.4]{tunnell:llc}.  
Hence any non-minimal representation has even conductor, and thus even level.

\subsubsection*{1-dimensional representations}

Any 1-dimensional representation of $B^\times$ is of the form
$\pi = \mu \circ N$ for some character $\mu$ of $F^\times$.  Then
$\omega_\pi = \mu^2$.  Recalling that $N(\calU^n) = \fraku^{\lceil n/2 \rceil}$
we see that $\ell(\pi) = \max \{ 2c(\mu) - 2, 0 \}$.
Since $\pi' = \St \otimes \mu$, where $\St$ is the Steinberg representation,
we see that $c(\pi)=c(\pi') = \max \{ 2c(\mu), 1 \}$.  Hence 
$\ell(\pi) = \max \{ c(\pi) - 2, 0 \}$.  

Note if $c(\pi) = 1$, then $\mu$ is unramified so $\pi$ is trivial on $\calU^0$.
Otherwise, $\ell(\pi) = c(\pi) - 2$.  Hence we always have that $\pi$ is trivial on
$\calU^{c(\pi) - 1}$, as asserted above.

\subsubsection*{Nonabelian representations}

Now we describe the higher-dimensional representations of $B^\times$
in terms of induction from certain finite-index subgroups.

Suppose $\dim \pi > 1$.
Since $B^\times/\calO_B^\times \simeq \Z$,
$\pi$ must be nontrivial on $\calO_B^\times$, so $c(\pi) \ge 2$.
In fact, $\pi'$ is supercuspidal.  Let $\ell=\ell(\pi)=c(\pi)-2$, so
$\calU^{\ell+1}$ is the largest unit group in $\ker \pi$.
Recalling that $\calU^{\lceil (\ell+1)/2 \rceil}/\calU^{\ell+1}$ is abelian,
$\pi$ restricted to $\calU^{\lceil (\ell+1)/2 \rceil}$ must break up as a sum
of characters.

Fix an additive character $\psi : F \to \C^\times$ of level 1, i.e., $\psi$ is trivial on
$\frakp$ but not $\frako_F$.  For $\alpha \in B$, define
\[ \psi_\alpha(x) = \psi(\tr (\alpha(x-1))), \quad x \in B. \]
By essentially the same calculation as in \eqref{eq:abel-quo},
we can see that $\psi_\alpha$ is a character of $\calU^{\lceil (m+1)/2 \rceil}$
which is nontrivial on $\calU^{m}$ but trivial on $\calU^{m+1}$, where
$m=-v_B(\alpha) > 0$.  Furthermore, any character of 
$\calU^{\lceil (m+1)/2 \rceil} / \calU^{m+1}$ is of the form $\psi_\alpha$ for
some $\alpha \in \frakP^{-m}$.

We say $\alpha \in B^\times$ is minimal if it has odd valuation or
if the characteristic polynomial of $\varpi^{-v_B(\alpha)/2}\alpha$ is irreducible
mod $\frakp$.

\subsubsection*{Odd-level representations}

First suppose $c(\pi) = 2n+1 \ge 3$, i.e., $\ell(\pi) = 2n-1$.  Necessarily, 
$\pi$ is minimal. 
Since we may view $\pi |_{\calU^n}$ as a representation of the
abelian group $\calU^n/\calU^{2n}$, it is a sum of characters 
$\psi_\alpha$ for some collection of $\alpha \in \frakP^{1-2n}$.
By Mackey theory and normality of $\calU^n$, 
all of these $\psi_\alpha$'s are conjugate.  Fix one such $\alpha$.
Then $v_B(\alpha) = 1-2n$.

Let $E=F[\alpha] \subset B$, which is a ramified quadratic extension of $F$.
Then conjugation by some $x \in B^\times$ fixes $\psi_\alpha$ if and only if
$x \in J := E^\times \calU^n$.  From this, one can deduce that 
$\pi \simeq \Ind_J^{B^\times} \Lambda$ for some character $\Lambda$ of
$J$ such that $\Lambda|_{\calU^n} = \psi_\alpha$.

Since $B^\times/J$ has cardinality $|\calO_B^\times/\frako_E^\times \calU^n|
= q^{n-1}(q+1)$, this is the dimension of $\pi$.

\subsubsection*{Even-level representations}

Now suppose $c(\pi) = 2n \ge 2$, i.e., $\ell(\pi) = 2n-2$.  We further assume
$\pi$ is minimal.  As above, $\pi |_{\calU^n}$ is a representation of
the abelian group $\calU^n/\calU^{2n-1}$.  
Assuming $n \ge 2$, the restriction to
$\calU^{n}$ contains a character $\psi_\alpha$, where $\alpha$ is minimal 
of valuation $2-2n$.

Let $E$ be the unramified quadratic extension of $F$.
For any $n \ge 1$, set $J=E^\times \calU^{n-1}$.  In this case, we
can write $\pi \simeq \Ind_J^{B^\times} \Lambda$ where $\Lambda$
is a representation of $J/\calU^{2n-1}$ of dimension 1 (resp.\ $q$) 
if $n$ is odd (resp.\ even).  Specifically, if $n \ge 2$, then $\Lambda|_{\calU^{n}}
\simeq c \psi_\alpha$, for some $\alpha \in E$ as above, where $c$ is 1 or $q$
according to the parity of $n$ (see \cite[Section 54.7]{BH}).  If $n=1$ (i.e., $\pi$ has level zero),
we may take $\Lambda$ to be a character of 
$J = E^\times \calU^1 = E^\times \calO_B^\times$ which is trivial
on $\calU^1$ with $\Lambda|_{E^\times}$ a character of $E^\times$ of conductor
1.  Hence for all $n \ge 1$, we have that $\Lambda|_{E^\times}$ is a sum of characters
which are nontrivial on $\fraku_E^{n-1}$.

Again, computing $|B^\times/J|=2|\calO_B^\times
/\frako_E^\times \calU^{n-1}| = 2q^{2\lfloor (n-1)/2 \rfloor}$ (times $\dim \Lambda$)
gives the dimension of $\pi$, and we summarize the even and odd conductor
cases together for future reference:

\begin{lemma} \label{lem:dimpi}
Suppose $\pi$ is a minimal representation of $B^\times$
of dimension $> 1$.  If $c(\pi) = 2n+1$, then $\dim \pi = q^{n-1}(q+1)$.
If $c(\pi) = 2n$, then $\dim \pi = 2q^{n-1}$.
\end{lemma}

This is a special case of \cite[Proposition 6.5]{carayol}, which also covers higher
degree division algebras.

\begin{rem} Since the formal degree $d(\pi')$ of a discrete series representation 
$\pi'$ of $\GL_2(F)$, normalized so that $d(\St) = 1$, is simply $\dim \pi$
when $\pi' = \JL(\pi)$, the above lemma gives a formula for the formal degree
of a supercuspidal representation of $\GL_2(F)$ in terms of the conductor of
a minimal twist.  
\end{rem}
 
%
%

\section{Local new and old forms}
\label{sec:local-new}

We keep the local notation of the previous section.  

Let $E/F$ be a quadratic extension, $r \ge 1$, and consider a special order
$\calO = \calO_r(E)$.  Let $\pi$ be a smooth irreducible representation
of $B^\times$ with central character $\omega$.
Let $\Omega$ be a character of
$\calO^\times = \frako_E^\times \calU^{r-1}$ which is trivial on
$\calU^{r-1}$.  Denote by $c_E(\Omega)$ the conductor of $\Omega|_{\frako_E^\times}$.
Necessarily, $\Omega$ is trivial on $\calU^{r-1} \cap E^\times =
\fraku_E^{\lceil e(r-1)/2 \rceil}$ where $e=e(E/F)$, i.e., $c_E(\Omega) \le \lceil e(r-1)/2 \rceil$.
Consider the subspace of $\Omega$-equivariant vectors
\begin{equation}
 \pi^\Omega = \{ v \in \pi : \pi(g)v = \Omega(g) v, \, g \in \calO^\times \}.
\end{equation}
The goal of this section is to determine $\dim \pi^\Omega$.
We first note the following obvious necessary condition for existence of
equivariant vectors.

\begin{lemma} \label{lem:cond-bound}
We have $\pi^\Omega = 0$ unless $\Omega$ and $\omega$
agree on $\frako_F^\times$ and $c(\pi) \le r$.
\end{lemma}

\begin{proof} Compatibility of $\Omega$ and $\omega$ is obviously necessary
for $\pi^\Omega \ne 0$.  If $c(\pi) > r$, then $\pi$ is nontrivial on
$\calU^{r-1}$.  Then irreducibility implies $\pi$ cannot have any 
vectors fixed by the normal subgroup $\calU^{r-1}$.
\end{proof}

From now on we assume $\Omega|_{\frako_F^\times} = \omega|_{\frako_F^\times}$.
If we also have $c(\pi) \le r$, then we simply have
\begin{equation} \label{eq:piOm2}
 \pi^\Omega = \{ v \in \pi : \pi(t) = \Omega(t) v, \, t \in \frako_E^\times \}.
\end{equation}

We also have the following easy bound on dimension.

\begin{lemma} \label{lem32}
We have $\dim \pi^{\Omega} \le e(E/F)$.
\end{lemma}

\begin{proof} It is well-known that multiplicity one holds for $(B^\times, E^\times)$,
meaning in the irreducible decomposition of $\pi|_{E^\times}$, each character $\chi$ of 
$E^\times$ occurs with multiplicity $m(\pi, \chi)$ at most one.  Furthermore, each $\chi$ appearing
must satisfy $\chi|_{F^\times} = \omega$.

We may as well assume $c(\pi) \le r$, so then we have \eqref{eq:piOm2}.  Consequently
$\pi|_{E^\times}$ acting on the subspace $\pi^\Omega$ is simply the sum over
characters $\chi$ such that $\chi|_{\frako_E^\times} = \Omega|_{\frako_E^\times}$
and $m(\pi,\chi) = 1$.  If $E/F$ is unramified, $F^\times \frako_E^\times = E^\times$
so the compatibility with $\omega$ and $\Omega$ determines $\chi$ uniquely.
If $E/F$ is ramified, there are two possible $\chi$ which are compatible with both
$\omega$ and $\Omega$---these are determined by choosing $\chi(\varpi_E)$
so that $\chi(\varpi_E)^2 = \omega(\varpi)$, 
assuming $\varpi_E$ is chosen to square to $\varpi$.
\end{proof}

Let $\chi$ be a character of $E^\times$ which is compatible with
$\omega$ and $\Omega$ as in the above proof.  By Tunnell \cite{tunnell:eps}
and Saito \cite{saito}, we know $\chi$ occurs in $\pi|_{E^\times}$
(and thus contributes a line to $\pi^\Omega$ if $c(\pi) \le r$) if and only if 
$\eps(1/2,\pi_E \otimes \chi) = -\omega(-1)$, where $\pi_E$ denotes the base
change of $\pi$ to $(B \otimes E)^\times \simeq \GL_2(E)$.  Hence we have the formula
\begin{equation} \label{eq:eps-form}
 \dim \pi^\Omega = - {\omega(-1)} \sum_{i=1}^{e(E/F)} 
\frac{\eps(1/2, \pi_E \otimes \chi_i) - \omega(-1)}2
\end{equation}
when $c(\pi) \le r$ and $\chi_i$ runs over compatible characters $\chi$.

The local root numbers $\eps(1/2,\pi_E \otimes \chi)$ were calculated in \cite{tunnell:eps}
when $\pi'$ is a dihedral supercuspidal representation.  (In the case of odd residual
characteristic, all supercuspidals $\pi'$ are dihedral---\cite{saito}
reproved Tunnell's main result
without computing local root numbers in a way that also works in even characteristic.)
For minimal representations,
there 4 basic situations, according to whether the level of $\pi$ is odd or even
(i.e., the inducing subgroup $J=J_\pi$ for $\pi$ as in \cref{sec:loc-rep} contains a ramified or unramified quadratic extension) and whether
$E/F$ is unramified or ramified.  Half of the time (when the quadratic
extension contained in $J_\pi$ has opposite ramification type as $E/F$) the description of
the characters in $\pi|_{E^\times}$ is simple and depends only on conductors (and
compatibility with $\omega$), and thus $\eps(1/2, \pi_E\otimes \chi)$ is easily
described.  However, the other half of the time, the description of $\eps(1/2, 
\pi_E \otimes \chi)$ is complicated.

So instead of trying to use \eqref{eq:eps-form} to compute $\dim \pi^\Omega$,
we will examine $\pi|_{\frako_E^\times}$ directly using the description of $\pi$
as $\Ind_J^{B^\times} \Lambda$ for suitable $J$, $\Lambda$.  Things are somewhat
simplified by the fact that for our applications we do not need to consider
arbitrary ramification of $\Omega$ (and to some extent $\omega$).  In 
addition,
it seems that the description of $\pi|_{E^\times}$ is simpler when given in terms of the inducing data $(J,\Lambda)$
rather than the description of $\pi'$ as dihedrally induced.  (It also has the advantage
of being applicable when $\pi'$ is not dihedral.)  Hence this approach provides an alternate
way to compute root numbers $\eps(1/2, \pi_E \otimes \chi)$ via Mackey theory.
This is similar in spirit to the use of Mackey theory on $\GL_2(F)$ in \cite[Section 5]{FMP}, 
though the main goal there was determination of test vectors.
  That said, we will 
focus on computing $\dim \pi^\Omega$ when the description is simple enough to give
clean global statements (essentially, when it only depends upon conductors),
but see \cref{rem:eps-E-unram} for more discussion about this.

For 1-dimensional representations, the following is clear.

\begin{lemma} \label{lem:1-d}
Suppose $\pi = \mu \circ N_{B/F}$ and $\calO = \calO_r(E)$.  
Then $\pi^\Omega \ne 0$
if and only if $c(\pi) \le r$  
 and $\Omega|_{\frako_E^\times} = (\mu \circ N_{E/F}) |_ 
{\frako_E^\times}$.
\end{lemma}

Next we observe non-minimal representations often do not have
$\Omega$-equivariant vectors.

\begin{prop} \label{prop:loc-non-min}
Suppose $\pi$ is non-minimal, $\dim \pi > 1$,
and write $\pi = \tau \otimes \mu$ where $\tau$ is minimal.  Let $t=t(E/F)$.
Then $\pi^\Omega = 0$ if
one of the following conditions is satisfied:

\begin{enumerate}[(i)]

\item $E/F$ is unramified and $c_E(\Omega) < \frac{c(\pi)}2$; or

\item $E/F$ is ramified, $q$ is odd, and $c_E(\Omega) < c(\pi) - 1$; or

\item $E/F$ is ramified,  $c_E(\Omega) < c(\pi) - t - 1$,
$c(\pi) \ge 2t+4$, and $c(\pi) - c(\tau) > t$.

\end{enumerate}
\end{prop}

\begin{proof}
Recall  $c(\pi) = 2 c(\mu) > c(\tau)$.
Note that $\pi^\Omega = \tau^{\Omega \otimes \mu^{-1}}$, where
$\Omega \otimes \lambda = (\lambda \circ N_{B/F})|_{\calO^\times} \cdot \Omega$.
So for $\pi^\Omega \ne 0$ we need that $\Omega \otimes \mu^{-1}$ is trivial 
on $\calU^{c(\tau)-1} \subset \ker \tau$.  Hence it suffices to show that any
of the above conditions imply that $c_E(\Omega \otimes \mu^{-1})$
 is strictly bigger than $\lceil e(c(\tau) - 1)/2 \rceil$.

First, suppose $E/F$ is unramified.  Note $c(\mu^{-1} \circ N_{E/F}) = c(\mu)$.
So if $c_E(\Omega) < c(\mu)$,
then $c_E(\Omega \otimes \mu^{-1}) = c (\mu \circ N_{E/F}) = \frac{c(\pi)}2 >
\frac{c(\tau)}2 \ge \lceil (c(\tau) - 1)/2 \rceil$.

Now suppose $E/F$ is ramified.  Since $\dim \pi > 1$ and $\pi$ non-minimal
implies $c(\pi) \ge 4$, (ii) follows from (iii), so assume (iii) holds.  
By \eqref{eq:norm-ram}, $c(\mu^{-1} \circ N_{E/F}) =
2 c(\mu) - t - 1$ if $c(\mu) \ge t + 2$, which is equivalent to our assumption
$c(\pi) \ge 2t+4$.  Then the conditions on $c_E(\Omega)$ and $c(\pi) - c(\tau)$ imply
$c_E(\Omega \otimes \mu^{-1}) = c_E(\mu^{-1} \circ N_{E/F}) > c(\tau) - 1$.
\end{proof}

\begin{rem} 
The conditions in the lemma are in fact necessary.
For instance, suppose $E/F$ is unramified, and $\omega_\tau = 1$
and $c(\tau)$ is odd.  If we take $\Omega = \mu \circ N |_{\calO^\times}$,
then $\pi^\Omega = \tau^{\calO^\times}$, 
which we will see is 1-dimensional in \cref{thm:34}.

Moreover, we really do need the more complicated condition (iii)---that 
$\pi$ be sufficiently more ramified than $\tau$---when
$E/F$ is ramified and $q$ is even.  For instance, the 
global calculations in \cite[Examples 10.6, 10.7]{HPS:mams}
imply $\pi^\Omega$ can be nonzero when $E$ is a ramified quadratic
extension of $F=\Q_2$, $\Omega = 1$, $c(\pi) = 6$, and $c(\tau) = 5$.
\end{rem}

We say $\Omega$ is a minimally ramified extension of $\omega|_{\frako_F^\times}$
to $\calO^\times$ if $c_E(\Omega) = c(\omega)$ when $E/F$ is unramified and
$c_E(\Omega) = \max \{ 2c(\omega)-1, 0 \}$ when $E/F$ is ramified. 
It is easy to see that a smaller value of $c_E(\Omega)$ is not possible.
We will explain in \cref{sec:qmfs} how to construct a minimally
ramified extension $\Omega$ of $\omega$ under suitable bounds on $c(\omega)$.

\begin{thm} \label{thm:34}
Suppose $E/F$ is unramified, $\calO = \calO_{2r+1}(E)$,
$\pi$ is minimal, and $\Omega$ is a minimally ramified extension
of $\omega|_{\frako_F^\times}$.
Then $\pi^\Omega = 0$ unless $c(\pi) \le 2r+1$ and one of the following holds:

\begin{enumerate}[(i)]

\item $\dim \pi = 1$, in which case $\dim \pi^\Omega = 1$; or

\item $\dim \pi > 1$ and $c(\pi)$ is odd, in which case $\dim \pi^\Omega = 1$; or

\item $c(\pi)$ is even and $c(\omega) = \frac{c(\pi)}2$.

\end{enumerate}
\end{thm}

\begin{proof}
By \cref{lem:cond-bound}, we may assume $c(\pi) \le 2r+1$.
Then by  \eqref{eq:piOm2} and \cref{lem32}, 
$\dim \pi^\Omega = 1$ if $\Omega|_{\frako_E^\times}$ appears in 
$\pi|_{\frako_E^\times}$, and $\pi^\Omega = 0$ otherwise.

If $\dim \pi = 1$, then $\pi$ minimal means $\pi = \mu \circ N$ with
$\mu$ is unramified (so $\omega$  and thus $\Omega$ are unramified), and the statement is clear.  So assume $\dim \pi > 1$.  Let $\chi$ denote a character of $E^\times$ whose restriction to $F^\times$
is $\omega$.

First suppose $c(\pi) = 2n+1$, so $\dim \pi = q^{n-1}(q+1)$ by \cref{lem:dimpi}.
Note for $\chi$ to appear in $\pi|_{E^\times}$, we need $\chi$ to be trivial on
$\calU^{2n} \cap E^\times = \fraku_E^n$.  But since the number of $\chi$
compatible with $\omega$ such that $c(\chi) \le n$ is $q^{n-1}(q+1)$, by multiplicity
one we see that $\chi$ appears in $\pi$ if and only if $c(\chi) \le n$.  (This
argument is already in \cite{tunnell:eps}.)  Hence $\dim \pi^\Omega = 1$.

Now suppose $c(\pi) = 2n$.  Then we can write $\pi$ in the form
$\Ind_J^{B^\times} \Lambda$, where $J = E^\times \calU^{n-1}$, for a suitable
embedding of $E$ into $B$ and representation $\Lambda$ of $J$.
For $s \in B^\times $, put $J^s = sJs^{-1} \cap \frako_E^\times$
and $\Lambda^s(x) = \Lambda(s^{-1} x s)$.  By Mackey theory, we have
\begin{equation}
\label{eq:mackey}
 \pi|_{\frako_E^\times} = \bigoplus_{s \in \frako_E^\times \bs B^\times / J}
\Ind_{J^s}^{\frako_E^\times} \Lambda^s.
\end{equation}
Since $\calU^{n-1} \subset sJs^{-1}$ for any $s \in B^\times$, we have
$\calU^n \cap \frako_E^\times = \fraku_E^{\lceil n/2 \rceil}$ is contained in $J^s$.
In fact since $\frako_E^\times \calU^{n-1} = \frako_E^\times \calU^{n-2}$ when $n$ is
even, we have $\fraku_E^{m/2} \subset J^s$ where $m = n-2$ if $n$ is even
and $m = n-1$ if $n$ is odd.  Note $m$ is minimal such that $J = \frako_E^\times
\calU^m$.

Write 
\[ B = \{ \bmx x & \varpi y \\ \bar y & \bar x \emx : x, y \in E \}. \]
Thus $E^\times$ embeds diagonally in 
$B^\times$ and $\bmx 0 &\varpi \\ 1 & 0 \emx$ is a uniformizer for $B$. 
Also, $\calO_B$ is the set of elements of the above form with $x, y \in \frako_E$,
and $\calU^{2j}$ is the set of elements of the above form with $x \in 1 + \frakp_E^{j}$,
$y \in \frakp_E^{j}$ for any $j \ge 0$.
Hence, we can take a set of representatives for $B^\times / \frako_E^\times
\calU^m$ to be
\[ \{ \bmx 1 & \varpi y \\ \bar y & 1 \emx:  y \in \frako_E/\frakp_E^{m/2} \}
\cup \{ \bmx 1 & \varpi y \\ \bar y & 1 \emx \bmx 0 &\varpi \\ 1 & 0 \emx  :  y \in \frako_E/\frakp_E^{m/2-1} \}.
\]
(If $m=0$, we interpret $y \in \frako_E/\frakp_E^{-1}$ to mean $y=0$.)
It is easy to see this is also a set of representatives for $\frako_E^\times \bs B^\times / J$.

If $s =  \bmx 1 & \varpi y \\ \bar y & 1 \emx$, and $t \in E^\times$, we compute
\[ s \bmx t & \\ & \bar t \emx s^{-1} =
\frac 1{1-\varpi y \bar y} \bmx t- \varpi \bar t y \bar y & \varpi y( \bar t -  t) \\
\bar y(t - \bar t ) & \bar t - \varpi t y \bar y  \emx 
= \bmx t & \\ & \bar t \emx + \frac{(t- \bar t) y}{1-\varpi y \bar y} \bmx \varpi \bar y & -\varpi \\
-1 & \varpi y \emx. \]
If $s = \bmx 0 &\varpi \\ 1 & 0 \emx \bmx 1 & \varpi y \\ \bar y & 1 \emx$, the calculation
is the same as above with only the effect of exchanging $t$ and $\bar t$ on the right.

Since $\psi_\alpha$ is nontrivial on $\fraku_E^{n-1}$,
we see $\Lambda|_{\frako_E^\times}$ is a sum of characters which are nontrivial on 
$\fraku_E^{n-1}$ but trivial on $\fraku_E^{n}$.  
Note $t - \bar t \in \frakp^j$ for $t \in \fraku_E^j$. 
So the above calculation means 
either $\Lambda^s(1+x) = \Lambda(1+x)$ or $\Lambda^s(1+x) = \Lambda(1+\bar x)$ 
for $x \in \frakp_E^{n-1}$.  Hence 
any character appearing in $\pi|_{E^\times}$ has conductor $n$,
whence $\pi^{\Omega}$ can only be nonzero if $c_E(\Omega) = c(\omega) = n$.
\end{proof}

\begin{rem} \label{rem:gross}
Gross \cite{gross} calculated $\eps(1/2, \pi_E \otimes \chi) = (-1)^{c(\pi)}$
when $E/F$, $\chi$ and $\omega$ are all unramified.  This combined with
\eqref{eq:eps-form} gives an alternative proof of the above theorem in the case that $\omega$ is unramified (so $\Omega = 1$).  Conversely, our theorem
says that $\eps(1/2, \pi_E \otimes \chi) = (-1)^{c(\pi)}$ when $E/F$
is unramified, $\pi$ is minimal and $c(\chi) < \frac{c(\pi)}2$.
\end{rem}

\begin{rem} \label{rem:eps-E-unram}
The only case of the theorem where things are not completely settled
is (iii).  Suppose $c(\pi) = 2n$ and $\Lambda$ as in the proof.  Then 
$\Lambda|_{\fraku_E^{n-1}} \simeq c \lambda$ for some nontrivial character
$\lambda$ of $\fraku_E^{n-1}/\fraku_E^n$.  The proof shows that any $\chi$
appearing in $\pi|_{E^\times}$ agrees with $\lambda$ or $\tilde \lambda$ on $\fraku_E^{n-1}$,
and such $\chi$ must also restrict to $\omega$ on $F^\times$, where
$\tilde \lambda(x) = \lambda(\bar x)$.  There are 2 such characters
if $n=1$, and $2(q+1)q^{n-2}$ if $n \ge 2$.  On the other hand $\dim \pi = 2q^{n-1}$.
So when $n=1$ this completely characterizes the characters appearing in $\pi|_{E^\times}$
(namely $\lambda$ and $\tilde \lambda$), and when $n \ge 2$ it ``almost'' does.  
This supports the idea that $\eps(1/2, \pi_E \otimes \chi)$ may
 often have a simple description in terms of the inducing data $(J,\Lambda)$.
\end{rem}

\begin{thm} \label{thm:35} Suppose $E/F$ is ramified, $\calO = \calO_{2r}(E)$
for $r \ge 1$,
$\pi$ is minimal, and $\Omega$ is a minimally ramified extension
of $\omega|_{\frako_F^\times}$.
Then $\pi^\Omega = 0$ unless $c(\pi) \le 2r$, in which case we have:

\begin{enumerate}[(i)]
\item $\dim \pi^\Omega = 1$ if $\dim \pi = 1$; 

\item $\dim \pi^\Omega = 2$ if $\dim \pi > 1$ and $c(\pi)$ is even.
\end{enumerate}
\end{thm}

\begin{proof}
As before, we may assume $c(\pi) \le 2r$, and
let $\chi$ denote a character of $E^\times$ such that $\chi|_{F^\times} = \omega$.  
Again, the case $\dim \pi = 1$ is evident, so assume $\dim \pi > 1$.

First suppose $c(\pi) = 2n$, so $\ell(\pi) = 2n-2$.  For $\chi$ to appear in
$\pi|_{E^\times}$, we need $\chi$ to be trivial on $\calU^{2n-1} \cap \frako_E^\times
= \fraku_E^{2n-1}$, i.e., $c(\chi) \le 2n-1$.  The number of such $\chi$
which are compatible with $\omega$ is $2|\frako_E^\times/\frako_F^\times \fraku_E^{2n-1}|
= 2q^{n-1} = \dim \pi$.  Hence $\pi|_{E^\times}$ is simply the sum of all $\chi$
compatible with $\omega$ such that $c(\chi) \le 2n-1$.  (Again, this argument is
in \cite{tunnell:eps}.)  In particular, there are two such $\chi$ which agree with
$\Omega$ on $\frako_E^\times$, which gives (ii).
\end{proof}

\begin{rem} Globally, over $\Q$, one can compare class number formulas for
special orders from \cite{HPS:mams}
to dimensions of spaces of weight 2 newforms of level $p^{2n+1}$.
This comparison suggests, at least for $q$ odd and $\Omega = 1$,
that if $E/F$ is ramified with $\calO=\calO_{2r}(E)$, then $\dim \pi^\Omega = 1$
when $c(\pi) < 2r$ is odd.  When $q=2$, $\omega = \Omega = 1$ and 
$c(\pi) = 3 < 2r$, \cite[Proposition 3.7]{tunnell:eps} tells us indeed 
$\dim \pi^\Omega = 1$.
\end{rem}

%
%

\section{Quaternionic modular forms and Brandt matrices}
\label{sec:QMFs}

Let $F$ be a totally real number field of degree $d$, with integer ring
$\frako = \frako_F$, adele ring $\A=\A_F$,
and  infinite places $\nu_1, \ldots, \nu_d$.
Let $B$ a (totally) definite quaternion algebra over $F$, 
so each $B_{\nu_i}$ is isomorphic to Hamilton's quaternions $\mathbb H$.
Let $\frakD$ be the discriminant of $B$, which is the product over finite primes
$\frakp$ of $F$ at which $B$ ramifies.
Also fix a finite-order idele class character $\omega = \bigotimes \omega_v$ 
of $\A^\times$.  Then each $\omega_{\nu_i}$ is either trivial or the sign character
$\sgn$.

Let
$L^2(B^\times \bs B^\times(\A), \omega)$ be the space of square integrable (mod center) functions on $B^\times \bs B^\times(\A)$ which transform by $\omega$ on $\A^\times$.
Now $B^\times(\A)$ acts by right translation $R$ on 
$L^2(B^\times \bs B^\times(\A), \omega)$.  As $B^\times(\A)$ is compact mod
center, $R$ breaks up as a direct sum of irreducible representations, which are
the irreducible automorphic representations of $B^\times(\A)$ with central character 
$\omega$.

Consider an irreducible automorphic representation $\pi = \bigotimes \pi_v$
of $B^\times(\A)$ with central character $\omega$.  For $v | \infty$,
fix an embedding of $B_v^\times \simeq \mathbb H^\times$
into $\GL_2(\C)$, and for $k \ge 0$,
let $\Sym^{k}$ denote the composition of this embedding with
$k$-th symmetric power of the standard representation of $\GL_2(\C)$. 
Then any irreducible representation of
$B^\times_v$ is equivalent to some twist of $\Sym^{k_v}$, i.e., of
the form $\Sym^{k_v} \otimes (\lambda_v \circ N_{B_v/F_v})$, 
where $\lambda_v$ is a character
of $\R_{> 0}$, since all characters of $B^\times_v$ factor through the reduced norm.  
Note such a $\pi_v$ has central character $t \mapsto t^{k_{v}} 
\lambda_{v}(t^2)$, which must be $1$ or $\sgn$.  Since $\lambda_v$ is a character
of $\R_{> 0}$, this forces $\lambda_v(t) = t^{-\frac{k_v}2}$, and we see 
$\pi_v$ has central character $\omega_v = \sgn^{k_v}$.
We call $k_v$ the weight of
$\pi_v$, $\bfk = (k_{\nu_1}, \ldots, k_{\nu_d})$ the weight of $\pi$, and $\pi_\infty = 
\bigotimes_{v | \infty} \pi_v$ the infinity type of $\pi$.  
Denote this infinity type $\pi_\infty$ by $(\rho_{\bfk}, V_{\bfk})$.

We use the following notation at a finite place $v$: $\frako_v = \frako_{F,v}$
is the integer ring of $F_v$; $\varpi_v$ is a uniformizer in $F_v$;
 $\frakp_v$ is the associated prime ideal of $F$ (and by abuse of notation also the
 prime of $F_v$);
 $\calO_{B,v}$ is the unique maximal order of $B_v$ if $B_v$ is division and $M_2(\frako_{v})$
 otherwise; $\frakP_v$ the unique maximal ideal of $\calO_{B,v}$ if $B_v$ is division
 and $\varpi_v \calO_{B,v}$ else; and $\calU_v^r = 1 + \frakP_v^r$.
 In addition, to uniformize terminology, by a special order of level $\frakp_v^{r}$ when 
 $B_v$ splits, we simply mean
 an Eichler order of level $\frakp_v^{r}$, i.e., a conjugate of 
 $R_0(\frakp_v^r) = \{ \bmx a & b \\ c & d \emx \in  M_2(\frako_v) : c \equiv 0 \mod \frakp_v^r \}$.  

Denote by $\hat B^\times$ the finite part of $B^\times(\A)$, which we also
regard as the subgroup of $B^\times(\A)$ with trivial components at infinity (and will
do similarly for $\hat \calO^\times = \prod_{v < \infty} \calO_v^\times$ for an order $\calO$).
Let $K = \prod_{v < \infty} K_v$ be a compact open subgroup of $\hat B^\times$ such that
$K_v \simeq \GL_2(\frako_{F,v})$ for almost all $v$.  Let $\Omega = \bigotimes \Omega_v$
be a unitary extension of $\omega |_{\hat F^\times \cap K}$ to $K$ such that $\Omega_v$
is trivial for almost all $v$.  Let $\pi^\Omega$ be the subspace of
$\pi$ on which $K$ acts by $\Omega$, and similarly for $\pi_v^{\Omega_v}$, $v < \infty$.
Then $\dim \pi_v^{\Omega_v} < \infty$ for $v < \infty$, $\dim \pi_v < \infty$ for $v | \infty$,
and these dimensions are 1 for almost all $v$.  Hence $\dim \pi^\Omega < \infty$.  Moreover,
we can choose $K$ sufficiently small and $\Omega$ suitably so that $\pi^\Omega \ne 0$.  
In fact, by our local results
for $v < \infty$ we may take $K_v = \calO_v^\times$ for some  special order  $\calO_v$ 
of level $\frakp_v^{c(\pi_v)}$, assuming $\pi_v$ is minimal when $B_v$ is ramified.

Note that if $f \in \pi^\Omega$, then $\pi_\infty(g)f \in \pi^\Omega$ for all 
$g \in B^\times_\infty$.  Denote by $\pi_\bfk^\Omega$ the subspace of $\pi^\Omega$
consisting of the vectors of highest weight for $\pi_\infty$.  Given any nonzero
$f_0 \in \pi_\bfk^\Omega$, it generates an irreducible $\pi_\infty$-module $V_{f_0}$.
Since $(\pi_\infty, V)$ is self-dual,
we may identify the contragredient $(\pi_\infty^*, V_{f_0}^*)$ with $(\rho_\bfk, V_\bfk)$.
This identification is unique up to scaling.
Now fix some $x_0 \in \hat B^\times$ such that $f_0(x_0) \ne 0$.  
To $f_0$, we associate the map 
$\phi : B^\times \bs B^\times(\A) \to V_\bfk = V_{f_0}^*$ given by
\begin{equation} \label{eq:af2mf}
\langle \phi(x), f \rangle = f_0(x_0) f(x), \quad x \in B^\times(\A), \, f \in V_{f_0} 
\end{equation}
One readily sees that $\phi$ satisfies $\phi(x \alpha g) = 
\Omega(\alpha)\rho_{\bfk}(g^{-1}) \phi(x)$ for
$\alpha \in \hat \calO^\times$ and $g \in B^\times_\infty$.  
Taking such maps $f_0 \mapsto \phi$ on a basis of $\pi_\bfk^\Omega$ gives an
embedding of $\pi_\bfk^\Omega$ into the collection of 
$\phi : B^\times \bs B^\times(\A) \to V_\bfk$ satisfying this transformation property.

\subsection{Quaternionic modular forms} \label{sec:qmfs}

Now suppose $K = \hat {\calO}^\times$ with $\calO$ a special order of $B$
of level $\frakN = \prod \frakp_v^{r_v}$, by which we mean $\calO_v$ is 
special order of $B_v$ of level $\frakp_v^{r_v}$ for each $v < \infty$.
In \cref{sec:hecke-corr}, we will assume $\calO_v$ is of ``unramified quadratic
type'' when $v | \mathfrak D$ and $r_v$ is odd, but we do not impose this restriction
until then.

We want to
construct spaces of modular forms on $B$ which correspond to Hilbert
modular forms of level $\frakN$ and central character (i.e., nebentypus) $\omega$.
For such forms to exist, we need that $c(\omega_v) \le r_v$ for each $v < \infty$
and further that $c(\omega_v) \le \frac{r_v}2$ for $v | \frakD$,
so we assume this now.  (All forms which are primitive at primes $v | \frakD$
satisfy this by \cite[Theorem 6.8]{shemanske-walling}).
We extend $\omega|_{\hat \frako_F^\times}$ 
to a character $\Omega$ of $\hat \calO^\times$ as follows.  In fact,
for later use, we will extend $\Omega$ to be a semigroup homomorphism
from $\hat \calO$ to $\C$.

Let $v < \infty$.  On $\frako_v$, we may view $\omega_v$ as a character
of $\frako_v^\times/\fraku_v^{c(\omega_v)}$, which we may pullback to a 
multiplicative map $\frako_v/\frakp^{c(\omega_v)}$ by extending it to be 1 on
noninvertible elements if $\omega_v$ is unramified and 0 on 
$\frakp_v$ if $\omega_v$ is ramified.

If $B_v$ is split, we can realize $\calO_v =  R_0(\frakp_v^n)$.
Then set $\Omega_v( \bmx a & b \\ c & d \emx) = \omega_v(d)$.

If $B_v$ is division, we can realize $\calO_v = \frako_{E,v} + \frakP_v^{r_v-1}$,
where $E_v/F_v$ is a quadratic extension.  Assume $E_v/F_v$ is unramified if
$r_v = 1$.  We first extend $\omega_v$ to 
a semigroup homomorphism $\Omega_v$ of $\frako_{E,v}$ such that 
$c_{E_v}(\Omega_v) = c(\omega_v)$ if $E_v/F_v$ is unramified and
$c_{E_v}(\Omega_v) = \max \{ 2c(\omega_v) - 1, 0 \}$ if $E_v/F_v$ is ramified.
(This is always possible by Frobenius reciprocity as $\frako_{v}^\times/\fraku_v^m$ embeds 
in $\frako_{E,v}^\times/\fraku_{E,v}^{m'}$ where $m'=m$ or $\max \{ 0, 2m-1 \}$ according to whether
$E_v/F_v$ is unramified or ramified.)  Assume $\Omega_v = 1$ if $\omega_v$
is unramified.
If $E_v/F_v$ is unramified, then also $\calO_v = \frako_{E,v} + \frakP_v^{r_v}$
and for $x \in \frako_{E,v}$, $\alpha \in \frakP_v^{r_v}$ we define 
$\Omega_v(x + \alpha) = \Omega_v(x)$.  If $E_v/F_v$ is ramified, simply define
$\Omega_v(x + \alpha) = \Omega_v(x)$ for $x \in \frako_{E,v}$, $\alpha \in  
\frakP_v^{r_v-1}$.  In either case, this is well defined since $c(\omega_v) \le \frac{r_v}2$,
and makes $\Omega_v$ a minimally ramified extension of $\omega_v |_{\frako_v^\times}$.
 
Note that $\Omega_v$ is trivial on $\calU_v^{r_v}$, and in fact 
trivial on $\calU_v^{r_v-1}$ when $B_v$ is division.  Further $\Omega_v$
is trivial on $\calO_v^\times$ whenever $\omega_v$ is unramified.
Note the above definition of $\Omega$ actually makes it a semigroup
homomorphism $\Omega : \hat \calO \to \C$, which is important for defining
Hecke operators and Brandt matrices with character.

Let $\bfk = (k_{\nu_1}, \ldots, k_{\nu_d})$.
We define the space of quaternionic modular forms of weight $\bfk$, level $\calO$,
and character $\Omega$ to be
\begin{multline*}
 M_\bfk(\calO, \Omega)  = \left\{ \phi : B^\times \bs B^\times(\A)  \to V_{\bfk} \mid 
\phi(x \alpha g) =  \Omega(\alpha)\rho_{\bfk}(g^{-1}) \phi(x), \,
\right. \\ \left. \text{for all }  x \in B^\times(\A), \, \alpha \in \hat \calO^\times, \, g \in B^\times_\infty \right\}.
\end{multline*}
When the class number $h_F$ of $F$ is 1, then $\A^\times = F^\times \hat \frako^\times F_\infty^\times \subset B^\times \hat \calO^\times B_\infty^\times$, so all such forms
must transform under the center by $\omega$ (assuming $\omega_v = \sgn^{k_v}$ for
all $v | \infty$---if not this space must be 0).
In general, not all such forms will transform
on the center by $\omega$, but we define the subspace of those that do to be
\begin{multline*}
 M_\bfk(\calO, \Omega, \omega)  = \left\{ \phi : B^\times \bs B^\times(\A)  \to V_{\bfk} \mid 
\phi(z x \alpha g) = \omega(z) \Omega(\alpha)\rho_{\bfk}(g^{-1}) \phi(x), \,
\right. \\ \left. \text{for all }
z \in \A^\times, \, x \in B^\times(\A), \, \alpha \in \hat \calO^\times, \, g \in B^\times_\infty \right\}.
\end{multline*}
Viewing $M_\bfk(\calO, \Omega)$ as a representation space for $F^\times \bs \A^\times$,
we get a decomposition
\[ M_\bfk(\calO, \Omega) = \bigoplus_\omega M_\bfk(\calO, \Omega, \omega), \]
where $\omega$ runs over all (necessarily finite order) idele class characters of $F$ 
which agree with $\Omega$ on $\hat \frako^\times$ such that $\omega_v = \sgn^{k_v}$ for
all $v | \infty$.  We note there are at most $h_F$ such $\omega$.

When $\bfk = {\bf 0} = (0, 0, \cdots, 0)$, $V_{\bfk}$ is 1-dimensional,
and $\phi$ may factor through the (reduced) norm, i.e., $\phi = \mu \circ N$ for
some character $\mu$ of $F^\times \bs \A^\times$.  
Let $E_{\bf 0}(\calO,\Omega, \omega)$ be the linear span of such $\phi = \mu \circ N$,
which we think of  as the Eisenstein subspace of $M_\bfk(\calO, \Omega, \omega)$ (even 
though there are no cusps).  More explicitly, a basis of $E_{\bf 0}(\calO,\Omega, \omega)$
is given by the set of $\mu \circ N$ where $\mu$ ranges over idele class characters
such that $\mu^2 = \omega$ and the local components of $\mu \circ N$ and
$\Omega$ agree on each $\calO_\frakp^\times$.  
In particular, each such 
$\mu_\frakp$ is unramified when $\frakp \nmid \frakN$.  
Moreover, if $\omega$ (and thus
$\Omega$) is trivial, $\mu$ runs over the set of quadratic idele class characters (including
the trivial character) which are unramified at all finite places
(note the number of such $\mu$ equals the 2-rank of the narrow Hilbert class group of
$F$).

We can define an inner product on $M_{\bf 0}(\calO, \Omega, \omega)$ given by
\[ (\phi, \phi') = \int_{\A^\times B^\times \bs B^\times(\A)} \phi(x) \overline{\phi'(x)} \,
dx \]
for a suitable choice of Haar measure $dx$ on $B^\times(\A)$.  Let
$S_{\bf 0}(\calO, \Omega, \omega)$ be the orthogonal complement of 
$E_{\bf 0}(\calO,\Omega, \omega)$ in $M_{\bf 0}(\calO,\Omega, \omega)$ with respect to this
inner product, which we call the space of cusp forms.  
If $\bfk \ne {\bf 0}$, set $E_\bfk(\calO,\Omega, \omega) = 0$ and
$S_\bfk(\calO,\Omega, \omega) = M_\bfk(\calO,\Omega, \omega)$.

Using \eqref{eq:af2mf} to translate from automorphic forms to quaternionic modular
forms gives vector-space isomorphisms
\begin{equation} \label{eq:Mk-decomp}
 M_\bfk(\calO, \Omega, \omega) \simeq \bigoplus_\pi \pi_\bfk^\Omega,
\end{equation}
and
\begin{equation} \label{eq:Sk-decomp}
 S_\bfk(\calO, \Omega, \omega) \simeq \bigoplus_{\dim \pi > 1} \pi_\bfk^\Omega,
\end{equation}
where $\pi$ runs over equivalence classes of irreducible automorphic 
representations of $B^\times(\A)$ with central character $\omega$.  
We will often identify $\pi_\bfk^\Omega$ with the corresponding subspace of 
$M_\bfk(\calO,\Omega,\omega)$ obtained from \eqref{eq:af2mf}.
The dimensions of the left-hand sides are finite (e.g., from finiteness of the class
number and the description below), so there are only finitely many
nonzero $\pi^\Omega$'s appearing on the right-hand side. In essence,
the main goal of this paper is an explicit description of these decompositions
for special orders $\calO$.  

Note one gets analogous
decompositions of $M_\bfk(\calO, \Omega)$ and $S_\bfk(\calO, \Omega)$
by summing the decompositions in \eqref{eq:Mk-decomp} and
\eqref{eq:Sk-decomp} over suitable $\omega$.

We may identify the (invertible, i.e., locally principal) fractional
right $\calO$-ideal classes with the set
$\Cl(\calO) = B^\times \bs \hat B^\times / \hat \calO^\times$.
Now fix a set of representatives $x_1, \ldots, x_h \in \hat B^\times$ for
$\Cl(\calO)$.  Since any $\phi \in M_\bfk(\calO, \Omega)$ is determined
by its values on $x_1, \ldots, x_h$, we can view 
\[ \phi : \{ x_1, \dots, x_h \} \to V_\bfk, \]
which must satisfy the compatibility condition: if $\gamma x_i =  x_i \alpha g$
for some $\gamma \in B^\times$, $\alpha \in \hat \calO^\times$
and $g \in B^\times_\infty$ (necessarily $g = \gamma_\infty$), then 
\begin{equation} \label{eq:phi-transf}
\phi(x_i) = \Omega(\alpha) \rho_\bfk(g^{-1}) \phi(x_i).
\end{equation}
Put $\kappa = \dim V_\bfk = \sum (k_i + 1)$.
Then we can view $\phi \in M_\bfk(\calO, \Omega)$ as the element
$(\phi(x_1), \cdots, \phi(x_h)) \in \C^{\kappa h}$.

Let $V_\bfk^{\Gamma_i, \Omega}$ denote the set of $v \in V_\bfk$ such that
$\rho_\bfk(\gamma) v = \Omega(\alpha) v$ for all
$\gamma \in \Gamma_i = x_i \hat \calO^\times B_\infty^\times x_i^{-1} \cap B^\times$,
where  $\alpha = x_i^{-1} \gamma \gamma_\infty^{-1} x_i \in \hat \calO^\times$ 
 and $\rho_\bfk(\gamma)$ means $\rho_\bfk(\gamma_\infty)$.  Then \eqref{eq:phi-transf}
implies each $\phi(x_i) \in V_\bfk^{\Gamma_i, \Omega}$, which is just the
set of $\Gamma_i$-invariant vectors in $V_\bfk$ if $\Omega$ is unramified.
 This implies that we have a vector-space isomorphism 
\[ M_\bfk(\calO,\Omega) \simeq \bigoplus_{i=1}^{h} V_\bfk^{\Gamma_i, \Omega}. \]
Namely, for $1 \le i \le h$, we associate with an element 
$v_i \in V_\bfk^{\Gamma_i, \Omega}$ the unique element of $M_\bfk(\calO,\Omega)$ 
such that $\phi(x_i) = v_i$ and $\phi(x_j) = 0$ for $j \ne i$.

\subsection{Hecke operators}

Now we define Hecke operators, whose action will be given by
Brandt matrices (with character) operating on $M(\calO, \Omega) \simeq 
\bigoplus V_\bfk^{\Gamma_i, \Omega} \subset \C^{h \kappa }$.

For $v < \infty$, 
let $\calO_v^\bullet = \calO_v \cap B_v^\times$.  We have chosen $\Omega$
so that it is a semigroup homomorphism from $\hat \calO^\bullet$ to $\C$,
where $\hat \calO^\bullet$ is the set of $(\alpha_v) \in \hat B^\times$ such that
$\alpha_v \in \calO_v^\bullet$ for all $v < \infty$.  
(Our prescription of $\Omega$ is not as specific as that in 
\cite{HPS:crelle}, \cite{HPS:mams}.)

For $\phi \in M_\bfk(\calO, \Omega)$ and $\alpha \in \hat \calO^\bullet$,
we define the double coset operator $\hat \calO^\times \alpha \hat \calO^\times$ by
\begin{equation} \label{eq:dc-oper}
(\hat \calO^\times \alpha \hat \calO^\times \cdot \phi )(x) = \sum_\beta \Omega^{-1}(\beta) \phi(x \beta),
\quad \text{where } \hat \calO^\times \alpha \hat \calO^\times = 
\bigsqcup_\beta \beta \hat \calO^\times.
\end{equation}
Here by $\Omega^{-1}(\beta)$ we mean $\frac 1{\Omega(\beta)}$ if
$\Omega(\beta) \ne 0$ and $0$ if $\Omega(\beta) = 0$.
Note that if $\beta' = \beta u$ where $u \in \hat \calO^\times$,
then $\Omega^{-1}(\beta') \phi(x \beta') = \Omega^{-1}(\beta) \phi(x \beta)$, 
so \eqref{eq:dc-oper} does not depend upon the choice of 
decomposition $\hat \calO^\times \alpha \hat \calO^\times = \bigsqcup_\beta \beta 
\hat \calO^\times$. 

For a nonzero integral ideal
$\frakn$, let $ T_\frakn =  T_\frakn^\calO$ denote the formal sum of (distinct) double cosets
$\hat \calO^\times \alpha \hat \calO^\times \subset \hat B^\times$, where $\alpha$ 
runs over elements  of $\hat \calO^\bullet $ such that 
$N(\alpha) \in \hat F^\times$ corresponds to the integral ideal $\frakn$, i.e., such that
$N(\alpha) \hat \frako \cap F = \frakn$.
In the obvious way, we may view $ T_\frakn$ as a (Hecke) operator on 
$M_\bfk(\calO, \Omega)$ via \eqref{eq:dc-oper}.

Note for $\alpha = (\alpha_\frakp) \in \hat \calO^\bullet$, 
$\calI = \bigcap \alpha_\frakp \calO_\frakp$ is an integral (nonzero locally principal)
right $\calO$-ideal.
  Conversely, given any integral (nonzero locally principal) right
$\calO$-ideal $\calI$, we can write each $\calI_\frakp = \alpha_\frakp \calO_\frakp$,
where $\alpha_\frakp \in \calO_\frakp/\calO^\times_\frakp$.  Thus the integral right
$\calO$-ideals are in one-to-one correspondence with $\hat \calO^\bullet/ \hat \calO^\times$.
Hence we have
\begin{equation} \label{eq:hecke-oper}
 ( T_\frakn \phi)(x) = \sum_{\beta \in \hat \calO(\frakn)/ \hat \calO^\times} \Omega^{-1}(\beta)\phi(x \beta), \quad \text{where } 
 \hat \calO(\frakn) = \{ \alpha \in \hat \calO^\bullet: N(\alpha) \hat \frako \cap F = \frakn \}.
 \end{equation}
In particular, if $\omega$ is unramified so $\Omega$ is trivial on $\hat \calO$, 
we can just interpret $T_\frakn \phi$ as the sum of
right translates of $\phi$ by integral right $\calO$-ideals of norm $\frakn$.
We always have that $T_\frako$ acts trivially on $M_\bfk(\calO,\Omega)$.

The full Hecke algebra $\mathcal H(\calO,\Omega)$ for $M_\bfk(\calO, \Omega)$
is the algebra over $\C$ generated by all $T_\frakn$'s.  The unramified Hecke algebra
$\mathcal H^S = \mathcal H^S(\calO,\Omega)$ is the subalgebra generated by the $T_\frakn$'s
for $\frakn$ coprime to all prime ideals $\frakp$ such that $\calO_\frakp
\not \simeq M_2(\frako_\frakp)$ (implicitly, $S$ denotes the set of
such primes $\frakp$ together with the set of infinite places).  
Since $\mathcal H^S$ is a commutative algebra of normal 
operators, $M_\bfk(\calO,\Omega)$ has a basis of eigenforms
for $\mathcal H^S$.
Moreover, if $\phi \in M_\bfk(\calO, \Omega)$ is an eigenform for
$\mathcal H^S$, then $\phi \in \pi_\bfk^\Omega$ for some irreducible
$\pi$ in the decomposition \eqref{eq:Mk-decomp}.

Since $M_\bfk(\calO, \Omega)$ has a factorizable basis of eigenfunctions
$\phi = \bigotimes \phi_v$, we can decompose our global Hecke operators into
a product a local Hecke operators via
\begin{equation} \label{eq:hecke-fact}
 T_\frakn \phi =  \prod T_{\frakp^{v_\frakp(\frakn)}} \phi_\frakp,
\end{equation}
where we let 
$T_{\frakp^m}$ act on the local component $\phi_\frakp$ by 
the local analogue of \eqref{eq:hecke-oper}.  (We can view each local operator
$T_{\frakp^m}$ as acting on local representation spaces $\pi_\frakp$
and locally $\beta$ runs over
elements of $\calO_\frakp^\bullet/\calO_\frakp^\times$ such that $v_{B_\frakp}(\beta) = m$.)
When $\calO_\frakp \simeq M_2(\frako_\frakp)$, 
these local Hecke operators are the
usual unramified local Hecke operators for $\GL_2(F_\frakp)$.

The following is a quaternionic analogue of the calculation of ramified
Hecke eigenvalues for elliptic or Hilbert newforms (e.g., \cite[Theorem 3]{atkin-lehner}).

\begin{prop} \label{prop:hecke-comp1}
 Let $\pi$ an irreducible representation appearing in \eqref{eq:Mk-decomp}.  
Let $\frakp | \frakD$, $\frakp^{n_\frakp}$ be the
level of the local order $\calO_\frakp$, and $m \ge 1$.

\begin{enumerate}[(i)]

\item If $\omega_\frakp$ is ramified then $T_{\frakp^m} = 0$ on $\pi_\bfk^\Omega$.

\item If $\omega_\frakp$ is unramified, $n_\frakp = 1$, 
and $\pi_\frakp \simeq \mu_\frakp \circ N_{B_\frakp/F_\frakp}$ for an unramified
character $\mu_\frakp$ of $F_\frakp^\times$, then
$T_{\frakp^m}$ acts by $\mu_\frakp(\varpi_\frakp)^m$ on $\pi_\frakp$.

\item If $n_\frakp = c(\pi_\frakp) \ge 2$, then $T_{\frakp^m} = 0$
on $\pi^\Omega_\bfk$.

\item If $c(\pi_\frakp) \ge 3$ and $m \ge n_\frakp - c(\pi_\frakp) - 1$, then $T_{\frakp^m} = 0$
on $\pi^\Omega_\bfk$.

\end{enumerate}
\end{prop}

\begin{proof}
As explained above, the proposition boils down to a local calculation.
So for simplicity, for the proof,
we drop subscripts and revert to the notation of the local sections.
E.g., $F$, $B$, $\calO$, etc.\ now denote what were
earlier denoted by $F_\frakp$, $B_\frakp$, $\calO_\frakp$, etc.  In particular,
$\calO$ is a special order of level $\frakp^n$, $\Omega$
is a minimally ramified extension of (the restriction to $\frako^\times$ of) 
$\omega=\omega_\pi$ to $\calO^\times$, and 
$\Omega$ extends to a multiplicative function of $\calO$ such that $\Omega = 1$
if $\omega$ is unramified and $\Omega = 0$ outside $\calO^\times$ if $\omega$ is
ramified.  Consequently, $T_{\frakp^m} = 0$ if $\omega$ is ramified, and we may
assume $\omega$ is unramified and $\Omega = 1$.

Then the local Hecke operator is simply 
\[ (T_{\frakp^m} \phi) (x) = \sum_\beta \phi(x \beta), \]
where $\beta$ runs over the elements of $\calO/\calO^\times$ of norm $q^m$, 
i.e., $\beta$ runs over the right $\calO$-ideals of norm $\frakp^m$.

Suppose $\calO=\calO_B$ is the maximal order.  Then $\pi = \phi = \mu \circ N$
for some unramified character $\mu$ of $F^\times$ whose square is $\omega$.
Hence $(T_{\frakp^m} \phi) (x) = \mu(N(x \varpi_B^m)) = \mu(\varpi^m) \phi(x)$,
which gives the second part.

Now suppose $\calO$ has level $\frakp^n$, where $n \ge 2$. 

Consider an element $\beta \in \calO$ with $v_B(\beta) = m$.  Then 
$u \beta  \in \beta + \frakP^{m+m'}$ for any $u \in \calU^{m'}$.  In particular,
if we take $m' \ge \max \{ 1, n-m-1 \}$, then such $u \beta$ lies in $\calO$ with
valuation $m$.  Suppose $c(\pi) \ge 3$, $m \ge n - c(\pi) - 1$, and take $m' = c(\pi) - 2$.
Then the abelian group $\calU^{m'}/\calU^{n-1}$, which is 
nontrivial if $n \ge 3$, acts on the left on the right $\calO$-ideals of norm $\frakp^m$.
Now $\pi$ restricted to $\calU^{m'}$ breaks up as a sum of characters, all of which
are nontrivial. 
Assume $\phi$ is an eigenfunction of one of these characters, say $\psi'$.
Then $(T_{\frakp^m} \phi)(x)$ breaks up as a sum of expressions of the form
$\sum_u \psi'(u) \phi(x \beta)$, all of which are $0$.
In particular, for all $n \ge 3$ and $m \ge 1$, 
$T_{\frakp^m}$ annihilates all $\pi$ with $c(\pi) = n$.
This proves case (iv) and, in the situation that $n_\frakp \ge 3$, case (iii).

Finally, suppose $n=2$.  Then $\calO = \frako_E + \frakP$, where
$E/F$ is a (in fact any) ramified quadratic extension.  Since $\frakP \subset \calO$,
$\calO_B^\times$ acts (transitively) by left multiplication on the elements of $\calO$
of a fixed valuation $m \ge 1$.  Then, as above, if $c(\pi) = 2$, $(T_{\frakp^m} \phi)$
breaks up as a sum of sums of nontrivial characters on the abelian group 
$\calO_B^\times / \calU^1$, and thus kills $\pi$.
\end{proof}

\begin{rem} \label{rem:41}
Continuing with the local notation in the proof, if $\dim \pi = 1$, again we can
write $\pi = \phi = \mu \circ N$.  Suppose $\Omega = 1$.
Then by \cref{thm:34} and \cref{thm:35}, the only way
$\pi^\Omega \ne 0$ is if $\mu$ is trivial on $N_{E/F}(\frako_E^\times)$.  Now  
$(T_{\frakp^m}\phi)(x) = \sum \mu(N(\beta))  \phi(x)$, where $\beta$ runs over the 
elements of $\calO/\calO^\times$ of norm $q^m$.
This generalizes (ii) to $n_\frakp \ge 1$.
If $\mu$ is unramified, we see the local Hecke eigenvalue is just $\mu(q^m)$ times the 
number of $\calO$-ideals of norm $\frakp^m$. 
\end{rem}

We will also use the following calculation of ramified Hecke eigenvalues when
$B_\frakp$ is split.  Note, even though $B_\frakp \simeq M_2(F_\frakp)$, 
these local ramified Hecke operators are \emph{not} 
the standard ones that will arise in the next section, however this calculation
shows that they do agree on new forms of the appropriate level.
(A general comparison of these two definitions of ramified Hecke operators seems
not so simple.)

\begin{prop}
\label{prop:hecke-comp2}
Suppose $\frakp \nmid \frakD$, $n_\frakp \ge 1$ is the
level of the local Eichler order $\calO_\frakp$, and $m \ge 1$.
Let $\pi$ an irreducible representation appearing in \eqref{eq:Mk-decomp}
such that $c(\pi_\frakp) = n_\frakp$.  Then $T_{\frakp^m}$
acts by $\hat \calO^\times \bmx \varpi_\frakp^m & \\ & 1 \emx \hat \calO^\times$
on $\pi^\Omega$.
\end{prop}

\begin{proof}
As in the previous proof, this is really a local calculation, so 
we will just use local notation and drop the $\frakp$'s from our subscripts.
However, now $B \simeq M_2(F)$ is the split local quaternion
algebra, and we may take $\calO = R_0(\frakp^n)$.  
By assumption $\dim \pi^{\calO^\times}  = 1$.

Take a (nonzero) new vector $W \in \pi^{\calO^\times}$ in the Whittaker model
with respect to an additive character $\psi$ of order 0.
We know $W(1) = 1$, so it suffices to compute 
\[ T_{\frakp^m} W(1) = \sum_g \omega^{-1}(d) W (g), \]
where $g = \bmx a & b \\ c & d \emx$ runs over the elements in $\calO/\calO^\times$
with $v(\det g) = m$.  
Note that we can break up the set of $g = \bmx a & b \\ c & d \emx$ in 
$\calO$ with $v(\det g) = m$ into two disjoint subsets $X$ and $Y$, where
$X$ consists of such $g$ with $d \in \frako^\times$ and $Y$ consists of such $g$
with $d \in \frakp$.  For any such $g \in X$, we can write
\[ g = \bmx d^{-1} & bd^{-1} \\ & 1 \emx \bmx \varpi^m & \\ & 1 \emx
\bmx u & \\ c & d \emx \in \calO^\times \bmx \varpi^m & \\ & 1 \emx \calO^\times, \]
where $u \varpi^m = \det g$.  Hence $X = \calO^\times \bmx \varpi^m & \\ & 1 \emx \calO^\times$, and it suffices to show that there is no contribution from 
$g \in Y/\calO^\times$ to the above sum for $T_{\frakp^m} W(1)$.

Let $U = \{ \bmx 1 & x \\ & 1 \emx : x \in \frakp^{-1}/\frako \}$.  Note that
$U$ acts by left multiplication on $Y$.
Let $s_g$ be the size of the stabilizer in $U$ of $g \calO^\times$,
which only depends upon the $U$-orbit of $g$.  Then the contribution of 
$g \in Y/\calO^\times$ to $T_{\frakp^m} W(1)$ is
\[  \sum_{g \in Y/\calO^\times} \omega^{-1}(d) W (g) = \sum_{g \in U \bs Y /\calO^\times}  \frac 1{s_{g}}
\omega^{-1}(d) \sum_{x\in \frakp^{-1}/\frako}    W(\bmx 1 & x \\ & 1 \emx g). \]
The inner sum is $\sum \psi(x) W(g)$, which vanishes 
since $\psi$ is nontrivial on $\frakp^{-1}/\frako$.
\end{proof}

\subsection{Brandt matrices} \label{sec:brandt}

Here we explain how to describe $T_\frakn$ in terms of matrices.
Recall we have a fixed set of representatives $x_1, \ldots, x_h$ for $\Cl(\calO)$.
We also fix a basis $v_1, \ldots, v_\kappa$ of $V_\bfk$, and use this to realize
$\rho_\bfk(\gamma) \in \GL_\kappa(\C)$ for $\gamma \in B^\times$.
Regard any 
$\phi \in M_\bfk(\calO, \Omega)$ as a column vector 
${}^t(\phi(x_1) \, \cdots \, \phi(x_h)) \in \C^{h \kappa}$.
To simplify notation, we extend $\Omega$ trivially from $\hat \calO$ to
$\hat \calO \times B_\infty^\times$, i.e., $\Omega(\beta,g) = 
\Omega(\beta)$ for $\beta \in \hat \calO$, $g \in B_\infty^\times$.

For $1 \le i \le h$, let $\bar \Gamma_i = \Gamma_i / \frako^1$, where
$\frako^1$ denotes the group of (absolute) norm one units of $\frako$.
Since $\rho_\bfk$ is trivial on $\frako^1$, the finite group $\bar \Gamma_i$
acts on $V_\bfk$ via $\rho_\bfk$.  Then $V_\bfk$ has a canonical decomposition
into isotypic subspaces with respect to this action, with $V_\bfk^{\Gamma_i, \Omega}$
being one of these subspaces.  Put $e_i = | \bar \Gamma_i|$.  The map 
\[ \Xi_i(v) = \frac 1{e_i} \sum_{\gamma \in \bar \Gamma_i} \Omega(x_i^{-1} \gamma  x_i) \rho_\bfk(\gamma^{-1}) v \]
defines the orthogonal projection $\Xi_i : V_\bfk \to V_\bfk^{\Gamma_i, \Omega}$
with respect to the canonical decomposition.  Let $\Xi$ be the block diagonal
matrix with diagonal blocks $\Xi_1, \dots, \Xi_h$.  Then $\Xi$ is a linear projection
from $\C^{h \kappa}$ to $M_\bfk(\calO, \Omega)$, and for a nonzero integral ideal
$\frakn$, we define the Brandt matrix $A_\frakn$ to be the $h\kappa \times h \kappa$
matrix such that
\begin{equation} \label{eq:brandt}
 A_\frakn = T_\frakn \circ \Xi,
\end{equation}
as a linear operator from $\C^{h \kappa}$ to $M_\bfk(\calO, \Omega)$.

We now describe the entries of $A_\frakn$ in classical terms.
For $1 \le i \le h$, let $\calI_i = (x_i \hat \calO \times B_\infty) \cap B$ denote the right $\calO$-ideal
associated to $x_i$.  Then $\Gamma_i$ is the unit group of the 
left order $\calO_l(\calI_i) = (x_i \hat \calO x_i^{-1} \times B_\infty) \cap B$ of $\calI_i$.

\begin{prop} \label{prop:brandt}
Let $a_{ij}$ denote the $(i,j)$-th $\kappa \times \kappa$ block of the Brandt matrix $A_\frakn$.  Then
\begin{equation} \label{eq:brandt2}
a_{ij} = \frac 1{e_j} \sum_\gamma
\Omega^{-1}(x_i^{-1}  \gamma x_j) \rho_\bfk(\gamma), 
\end{equation}
where $e_j = [\calO_l(\calI_j)^\times : \frako^1]$ and $\gamma$ runs over a set of
representatives for
\[ \{  \gamma \in \calI_i \calI_j^{-1} : N(\gamma) \frako = \frakn N(\calI_i
\calI_j^{-1}) \}/\frako^1. \]
\end{prop}

\begin{proof}
For $1 \le i, j \le h$ and $\beta \in
\hat \calO^\bullet$, put
\[ \Gamma^{i,j}(\beta) = x_i \beta \hat \calO^\times B_\infty^\times x_j^{-1}  \cap B^\times. \]
Note $\Gamma^{i,j}(\beta)$ only depends upon the class of $\beta$ in 
$\hat \calO^\bullet/ \hat \calO^\times$.  
Conversely, $\Gamma^{i,j}(\beta) \cap \Gamma^{i,j}(\beta')  \ne \emptyset$ implies 
$\beta' \in \beta \hat \calO^\times$.
Moreover if $\Gamma^{i,j}(\beta) \neq \emptyset$, then $\Gamma^{i,j}(\beta) = \gamma_\beta \Gamma_j$ for any choice of $\gamma_\beta \in \Gamma^{i,j}(\beta)$.
Hence there is a bijection between the classes of $\beta \in \hat \calO^\bullet/ 
\hat \calO^\times$ such that $\Gamma^{i,j}(\beta) \ne \emptyset$ and the classes of
$\gamma_\beta \in (x_i \hat \calO^\bullet B_\infty^\times x_j^{-1} \cap B^\times)/ \Gamma_j$ via
$\Gamma^{i,j}(\beta) = \gamma_\beta \Gamma_j$.

Since $(T_\frakn \phi)(x_i) = \sum_\beta \Omega^{-1}(\beta)\phi(x_i \beta)$
where $\beta$ runs over $\hat \calO(\frakn)/ \hat \calO^\times$ as in \eqref{eq:hecke-oper},
for $v \in V_\bfk^{\Gamma_j,\Omega}$ we have
\[ a_{ij} v 
= \sum_\beta \Omega^{-1}(\beta u_\beta^{-1}) \rho_\bfk(\gamma_\beta) v
= \sum_\beta \Omega^{-1}(x_i^{-1} \gamma_\beta x_j) \rho_\bfk(\gamma_\beta) v, \]
where now $\beta$ runs over a set of representatives  of $\hat \calO(\frakn)/ \hat \calO^\times$ 
such that $\Gamma^{i,j}(\beta) \ne \emptyset$, and for each such $\beta$ we
choose a $u_\beta \in \hat \calO^\times$ and $\gamma_\beta \in \Gamma^{i,j}(\beta)$
such that $x_i \beta = \gamma_\beta x_j u_\beta \gamma_{\beta,\infty}^{-1}$.

Then the matrix $a_{ij}$ can be computed as the above operator (which depends
upon the choices of $\gamma_\beta$'s) times $\Xi_j$.  Multiplying these yields
\[ a_{ij} = \frac 1{e_j} \sum_\beta \sum_{\gamma \in\bar \Gamma_j} 
\Omega^{-1}(x_i^{-1} \gamma_\beta \gamma^{-1}  x_j) \rho_\bfk(\gamma_\beta
\gamma^{-1}). \]
Now applying the change of variable $\gamma_\beta \gamma^{-1} \mapsto \gamma$
gives \eqref{eq:brandt} with $\gamma$ running over a set of representatives of 
$(x_i\hat \calO(\frakn)B_\infty^\times x_j^{-1} \cap B^\times)/\frako^1$.
Since $(x_i \hat \calO x_j^{-1} \times B_\infty) \cap B = \calI_i \calI_j^{-1}$, the proposition follows.
\end{proof}

Note that in the special case of trivial weight and character ($\bfk = \zero$, $\Omega = 1$), 
we simply get
\[ a_{ij} = |\{ \gamma \in \calI_i \calI_j^{-1} : N(\gamma) \frako = 
\frakn N(\calI_i \calI_j^{-1}) \} / \calO_l(\calI_j)^\times |. \]

We also define a Brandt matrix $A_0$ associated to the zero ideal.
If $\bfk = \zero$ and $\Omega = 1$, let $A_0$ be the $h \times h$ diagonal matrix whose $i$-th diagonal entry
is $\frac 1{e_i}$.   Otherwise, set $A_0 = 0$.

The above approach to defining Brandt matrices is similar in spirit to that in 
\cite{HPS:mams}, which treats special orders with character over $\Q$.
Brandt matrices without character were also defined for maximal orders over 
totally real fields by Eichler in \cite{eichler:1977}.  Our definition coincides with
Eichler's (up to a formal difference of working with right ideals rather than
left ideals).

\subsection{Quaternionic new and old forms} \label{sec:quat-nfs}

For a special order $\calO' \supset \calO$, we say $\Omega'$ is an
admissible extension of $\Omega$ to $\hat \calO'$ if $\Omega'$ is an extension of
$\omega$ to $\hat \calO'$ constructed as in \cref{sec:qmfs} which agrees with 
$\Omega$ on $\hat \calO$.
We define the subspace $M_\bfk^\old(\calO, \Omega, \omega)$ of \emph{oldforms}
 in $M_\bfk(\calO, \Omega, \omega)$
to be the space generated by $\phi \in M_\bfk(\calO, \Omega, \omega)$
such that $\phi \in M_\bfk(\calO', \Omega', \omega)$, where $\calO'$ is a special
order properly containing $\calO$ and $\Omega'$ is an admissible extension of
$\Omega$ to $\hat \calO'$.
Define the subspace $M_\bfk^\new(\calO, \Omega, \omega)$
of \emph{newforms} to be the orthogonal complement of 
$M_\bfk^\old(\calO, \Omega, \omega)$  in $M_\bfk(\calO, \Omega, \omega)$
with respect to \eqref{eq:Mk-decomp}.

We define the new and old spaces of cusp forms and Eisenstein forms in 
$M_\bfk(\calO, \Omega, \omega)$ to be the intersection of the new and
old spaces with $S_\bfk(\calO, \Omega, \omega)$ and $E_\zero(\calO,
\Omega, \omega)$, and denote them in a similar way.

\begin{prop} \label{prop:quat-AL-decomp}

\begin{enumerate}[(i)]
\item We have the decomposition
\[ S_\bfk^\new(\calO, \Omega, \omega) \simeq \bigoplus \pi_\bfk^\Omega, \]
where $\pi$ runs over the representations in \eqref{eq:Sk-decomp}
such that $c(\pi) := \prod \frakp^{c(\pi_\frakp)} =  \frakN$.  

\item Moreover, 
as $\mathcal H^S(\calO,\Omega)$-modules, we have
\[ S_\bfk^\old(\calO, \Omega, \omega) \simeq \bigoplus 
S_\bfk^\new(\calO', \Omega', \omega), \]
where $\calO'$ runs over all special orders of $B$ properly containing $\calO$,
and $\Omega'$ runs over all admissible semigroup homomorphisms of 
$\hat \calO'$ extending $\Omega$.
\end{enumerate}
\end{prop}

Note that $\Omega'$ admissible implies that when
$B_\frakp$ is split (resp.\ ramified) $\Omega'_\frakp$ is trivial
on $\calU_\frakp^{r_\frakp'}$ (resp.\ $\calU_\frakp^{r_\frakp'-1}$),
where $\frakp^{r_\frakp'}$ is the
level of $\calO'_\frakp$.  It will follow from the proof below that, for each
such $\calO'$, there is always
at most one such admissible $\Omega'$.

\begin{proof}
Consider $\pi$ occurring in \eqref{eq:Sk-decomp}, and write $n_\frakp = c(\pi_\frakp)$.
By \cref{lem:cond-bound} and its well-known analogue for places where
$B$ splits, if $c(\pi) = \frakN$, then $\pi^\Omega_\bfk \subset S_\bfk^\new(\calO, \Omega, \omega)$.  This proves one direction of (i).

Now suppose $c(\pi) := \prod \frakp^{n_\frakp} \ne \frakN := \prod \frakp^{r_\frakp}$, so $n_\frakp \le r_\frakp$ for all $\frakp$ and this equality is strict for at least one $\frakp$.
Suppose $\frakp$ is such that $n_\frakp < r_\frakp$.

If $B_\frakp$ is split, we may take $\calO_\frakp = R_0(\frakp^{r_\frakp})$.  By local newform theory (see \cite{casselman})
we have that $\pi_\frakp^{\Omega_\frakp}$ is generated by the lines which
are $\Omega_\frakp'$-equivariant under the orders 
$\bmx \frako_\frakp & \frakp^{-j}_\frakp \\ \frakp^{n_v+j}_\frakp &
\frako_\frakp \emx$ where $0 \le j \le r_\frakp - n_\frakp$.  
Here there is a
unique admissible extension $\Omega_\frakp'$ of $\Omega_\frakp$ to each such superorder,
namely $\Omega_\frakp' \bmx a & b \\ c & d \emx = \omega_\frakp(d)$.
These superorders are precisely the local Eichler orders of level
$\frakp^{n_\frakp}$ containing $\calO_\frakp$.  

If $B_\frakp$ is a division algebra, then we can write $\calO_\frakp^\times = \calO_{r_\frakp}(E_\frakp)^\times = 
\frako_{E,\frakp}^\times \calU_\frakp^{r_\frakp - 1}$ where $E_\frakp/F_\frakp$ is a 
quadratic field extension.  Since $\pi_\frakp$ is trivial on 
$\calU_\frakp^{n_\frakp-1}$, $\Omega_\frakp$ must be trivial on 
$\calU_\frakp^{n_\frakp-1} \cap \calO_\frakp^\times$.
Note that since $\calO_{n_\frakp}(E_\frakp) =
\calO_{r_\frakp}(E_\frakp) \calU_\frakp^{n_\frakp-1}$, there is a unique
extension $\Omega_\frakp'$ of $\Omega_\frakp$ to $\calO_{n_\frakp}(E_\frakp)$
which is trivial on $\calU_\frakp^{n_\frakp-1}$
 Then $\pi_\frakp^{\Omega_\frakp}$ is the also set  of
$\Omega_\frakp'$-equivariant vectors for $\calO_{n_\frakp}(E_\frakp)^\times$ in $\pi$.  
Recall that any local order in $B_\frakp$
containing $\frako_{E,\frakp}$ is of the form $\calO_j(E_\frakp)$ for some $j$,
so $\calO_{n_\frakp}(E_\frakp)$ is the unique order of level $\frakp^{n_\frakp}$ containing $\calO_\frakp$.

Consequently, using the proven direction of (i), we see 
\[ \pi_\bfk^\Omega = \bigoplus S_\bfk^\new(\calO', \Omega', \omega), \]
where $\calO'$ runs over special orders of $B$ of level $c(\pi)$ containing
$\calO$ and $\Omega'$ is the unique extension of $\Omega$ to
$\hat \calO'$ where $\Omega_\frakp'$ is as described above when
$\calO'_\frakp \not \simeq \calO_\frakp$ and $\Omega_\frakp' = \Omega_\frakp$
otherwise.
This implies both the other direction of (i) and (ii).
\end{proof}

This yields the following Atkin--Lehner type decomposition.

\begin{cor} \label{cor:quat-AL-decomp2}
 As $\mathcal H^S(\calO,\Omega)$-modules, we have
\[ S_\bfk(\calO, \Omega, \omega) \simeq \bigoplus 
S_\bfk^\new(\calO', \Omega', \omega), \]
where $\calO'$ runs over all special orders of $B$ containing $\calO$,
and $\Omega'$ runs over all admissible semigroup homomorphisms of 
$\hat \calO'$ extending $\Omega$.
\end{cor}

For a special order $\calO$, we say the local order $\calO_\frakp$
is of \emph{unramified quadratic type} if $\calO_\frakp$ is an Eichler order
or is isomorphic to $\calO_r(E_\frakp)$ for some $r \ge 1$ where 
$E_\frakp/F_\frakp$ is the unramified quadratic extension.  Otherwise,
we say $\calO_\frakp$ is of \emph{ramified quadratic type}.  We say
 $\calO$ is globally of \emph{unramified quadratic type} if it is everywhere
locally.

Write the level $\frakN = \prod \frakp^{r_\frakp}$ of $\calO$ as $\frakN = \frakN_1 \frakN_2 \frakM$
where $\frakN_1$, $\frakN_2$ and $\frakM$ are the unique pairwise
coprime ideals such that $\frakM$ is coprime to $\frakD$
 and all primes $\frakp | \frakD$ at which $\calO$
is of unramified (resp.\ ramified) quadratic type divide $\frakN_1$ (resp.\
$\frakN_2$).  Necessarily, $r_\frakp$ is odd for $\frakp | \frakN_1$
and $r_\frakp \ge 2$ for $\frakp | \frakN_2$.

Consider a level $\frakN' | \frakN$, which we write as $\frakN' = \frakN_1' 
\frakN_2' \frakM'$, with $\frakN_1' | \frakN_1$, $\frakN_2' | \frakN_2$, and 
$\frakM' | \frakM$.  Write 
$\frakN' = \prod \frakp^{r_\frakp'}$.   For $\frakN'$ to be
the level of a special order $\calO'$ in $B$ containing $\calO$, we need 
$\calO'_\frakp$ to be of unramified quadratic type for $\frakp | \frakN_1'$
(so  $r_\frakp'$ is necessarily odd), which we assume.  
(For $\frakp | \frakN_2'$, $\calO'_\frakp$ is unramified quadratic type if $r_\frakp' = 1$
and of ramified quadratic type if $r_\frakp' \ge 2$.)
From the proof of the
above proposition, we find that the number of special orders $\calO'$
of level $\frakN'$ containing $\calO$ is
\[ m_\calO(\frakN') = \prod_{\frakp | \frakM} (r_\frakp - r'_\frakp + 1). \]
Since, for fixed $\frakN'$, all such $\calO'$ are conjugate, the 
isomorphism type (as a Hecke module) of the
subspace $M_\bfk(\calO', \Omega', \omega)$ in $M_\bfk(\calO, \Omega, \omega)$
only depends on $\frakN'$.

For each $\frakN' | \frakN$ as above, we fix a special order $\calO'(\frakN')$
of level $\frakN'$ containing $\calO$.  Then $\Omega$
has a (unique) admissible extension to 
$\hat \calO'$ if and only if $c(\omega_\frakp) \le \frac{r_\frakp'}2$
for each $\frakp | \frakN_1 \frakN_2$.
Consequently, the above corollary asserts a 
$\mathcal H^S(\calO, \Omega)$-module isomorphism
\begin{equation} \label{eq:AL-decomp2}
S_\bfk(\calO, \Omega, \omega) \simeq \bigoplus_{\frakN'} m_\calO(\frakN') 
S_\bfk^\new(\calO'(\frakN'), \Omega', \omega),
\end{equation}
where $\frakN' = \prod \frakp^{r_\frakp'}$ runs over divisors of $\frakN$ such that
(i) $r_\frakp' \ge \max\{ 1, 2c(\omega_\frakp) \}$ for all $\frakp | \frakN_1 \frakN_2$
and (ii) $r_\frakp'$ is odd for all $\frakp | \frakN_1$.

When $\bfk=\zero$, the above decompositions do not apply verbatim
to the Eisenstein spaces because the local representations are all 1-dimensional.
Instead, we simply have
\begin{equation} \label{eq:AL-eis-decomp}
E_\zero(\calO, \Omega, \omega) \simeq \bigoplus_{\frakN'} 
E_\zero^\new(\calO'(\frakN'), \Omega', \omega),
\end{equation}
with the sum as in \eqref{eq:AL-decomp2}.

\begin{lemma} \label{lem:eis-pM}
We have $E_\zero(\calO, \Omega, \omega) = 0$ unless $\omega_\frakp$
is unramified for all $\frakp \nmid \frakD$.
\end{lemma}

\begin{proof}
Consider $\mu \circ N \in E_\zero(\calO, \Omega, \omega)$.  
If $B_\frakp$ is split, for $\Omega_\frakp$ to agree
with $\mu_\frakp \circ N$ on 
$ \bmx \frako_\frakp^\times & \\ & 1 \emx \subset \calO_\frakp^\times$, 
we need $\mu_\frakp$, and thus $\omega_\frakp = \mu_\frakp^2$, 
to be unramified.
\end{proof}

The next proposition says that typically only ``small levels'' contribute
to the Eisenstein subspaces.

\begin{prop} 
Suppose $\calO$ has level $\frakN = \prod \frakp^{r_\frakp}$.  Then
$E_\zero^\new(\calO, \Omega, \omega) = 0$ if one of the following holds:

\begin{enumerate}[(i)]
\item $\frakM \ne \frako$;

\item $c(\omega_\frakp) < \frac {r_\frakp}2$ for some $\frakp^2 | \frakN_1$
or nondyadic $\frakp^3 | \frakN_2$; or

\item $c(\omega_\frakp) < \frac {r_\frakp}2$ for some $\frakp^3 | \frakN_2$
with $v_\frakp(\frakN_2) \ge 2 t(E_\frakp/F_\frakp) + 2$, where
$\calO_\frakp \simeq \calO_{r_\frakp}(E_\frakp)$.
\end{enumerate}

In particular, if $\calO$ is of unramified quadratic type and $\omega$
is unramified, then $E_\zero^\new(\calO, \Omega, \omega) = 0$ unless $\calO$ is maximal, in which case $E_\zero^\new(\calO, \Omega, \omega) = E_\zero(\calO, \Omega, \omega)$.
\end{prop}

\begin{proof}
Recall that $E_{\bf 0}(\calO,\Omega, \omega)$ is generated by
the characters $\mu \circ N$ of $B^\times(\A)$, 
where $\mu$ ranges over idele class characters
such that $\mu^2 = \omega$ and the local components of $\mu \circ N$ and
$\Omega$ agree on each $\calO_\frakp^\times$.  Thus, for a special
order $\calO' \supset \calO$ of level $\frakN'$, we have
$E_\zero(\calO, \Omega, \omega) = E_\zero(\calO', \Omega', \omega)$
(with $\Omega'$ as before) if $N((\calO'_v)^\times) = N(\calO_v^\times)$ for all finite $v$.  

Note that for the existence of a proper superorder $\calO' \supset \calO$ 
with an admissible extension $\Omega'$ of $\Omega$, we need that
$c(\omega_\frakp) < r_\frakp$ for some split $B_\frakp$ or 
$c(\omega_\frakp) < \frac {r_\frakp}2$ for some ramified $B_\frakp$.

If a local order $\calO'_v$ is of quadratic unramified type, then
$N((\calO_v')^\times) = \frako_v^\times$.  This proves 
$E_{\bf 0}^\new(\calO,\Omega, \omega) = 0$ if $c(\omega_\frakp) < r_\frakp$ for some
$\frakp | \frakM$, as well as the $\frakp^2 | \frakN_1$ part of (ii).  Then (i) follows from the previous lemma.

Since (iii) implies the rest of (ii), it remains to prove (iii).
Suppose $v | \frakN_2$
and $\calO_v \simeq \calO_{E_v}(r_v)$ where $E_v/F_v$ is some ramified
extension.   Write $t_v = t(E_v/F_v)$, and assume $r_v \ge 2 t_v + 3$ 
(i.e., $r_v \ge 3$ if  $v$ is non-dyadic).
Take $\calO_v' \simeq \calO_{E_v}(r_v-1)$.  Now
$N(\calO_v^\times) = N(\frako_{E_v}^\times) \fraku_v^{\lceil (r_v-1)/2 \rceil}$
and $N((\calO_v')^\times) = N(\frako_{E_v}^\times) 
\fraku_v^{\lceil (r_v-2)/2 \rceil}$.  Thus these norm subgroups are equal if
$N(\frako_{E_v}^\times) \supset \fraku_v^{\lceil (r_v-2)/2 \rceil}$.  By assumption
on $r_v$, this is the case since $N(\frako_{E_v}^\times) = N(\fraku_{E_v}^{t_v+1}) 
= \fraku_v^{t_v+1}$.
\end{proof}

%
%

\section{A classical Jacquet--Langlands correspondence}

Here we reinterpret the representation-theoretic Jacquet--Langlands correspondence
from \cite{JL} in more classical language, namely as a Hecke-module homomorphism
from a space of quaternionic modular forms to Hilbert modular forms.  
(Note: analytic details of the proof of the Jacquet--Langlands correspondence were
not actually completed in \cite{JL}, but for instance in \cite{duflo-labesse}---see also
\cite{gelbart} for an exposition.)  This interpretation
is essentially a generalization of \cite{shimizu} (see also \cite{gelbart})
from Eichler orders to special orders.

\subsection{Hilbert modular forms} \label{sec:HMFs}

First we recall some facts and set notation about Hilbert modular forms.
See \cite{shimura} for more details. We continue the notation from the 
previous section.  In particular, $F$ is a totally real number field of degree $d$
with adele ring $\A$. 

Let $\frako = \frako_F$ denote the integer ring of $F$, $\mathfrak d_F$ the
 absolute different, and $\frakN$ a nonzero integral
ideal of $\frako$.  Let $W=W(\frakN)$ (resp.\ $Y=Y(\frakN)$) be image under
the canonical involution $\iota$ of the level $\frakN$
subgroup of (resp.\ semisubgroup) of $\GL_2(\A)$ denoted by the same letter in 
\cite[Sec 2]{shimura}.  (These involuted subsets  make the subsequent
notation more straightforward.)
Namely, these are the subsets of $\GL_2(\A)$ with finite local components given by
\[ \begin{cases}
Y_\frakp = \bmx \frako_\frakp & \mathfrak d_F^{-1} \frako_\frakp \\
 \mathfrak d_F \frako_\frakp & \frako_\frakp \emx, \quad 
W_\frakp = Y_\frakp^\times &
\text{if } \frakp \nmid \frakN \\
Y_\frakp = \bmx \frako_\frakp & \mathfrak d_F^{-1} \frako_\frakp \\
\frakN \mathfrak d_F \frako_\frakp & \frako_\frakp^\times \emx, \quad 
W_\frakp = \bmx \frako_\frakp^\times & \mathfrak d_F^{-1} \frako_\frakp \\
\frakN \mathfrak d_F \frako_\frakp & \frako_\frakp^\times \emx = Y_\frakp^\times &
\text{if } \frakp | \frakN,
\end{cases} \]
and infinite components $\GL_2^+(\R)^d$.
Let $\psi$ be a finite-order Hecke character of $F$ of conductor dividing $\frakN$. We extend $\psi$ to $Y(\frakN)$ by $\psi_Y \bmx a & b \\ c & d \emx = \psi(d_\frakN \mod
\frakN)$, where $d_\frakN = \prod_{\frakp | \frakN} d_\frakp$ 
denotes the $\frakN$-part of $d$.

Let $\bfk = (k_1, \ldots, k_d)$.
We denote
by $M_\bfk(\frakN, \psi)$ the space of adelic holomorphic Hilbert modular
forms of level $\frakN$ and character $\psi$, which can be viewed
the space of functions $f$ on $\GL_2(\A_F)$ satisfying
\[ f(z \gamma g w) = \psi(z) \psi_Y(w) f(g), \quad
z \in \A_F^\times, \, \gamma \in \GL_2(F), \, w \in W(\frakN), \, w_\infty = 1, \]
together with the usual holomorphy conditions and weight $\bfk$ transformation
law at infinity.  Specifically, let $t_1, \ldots, t_{h_F}$ denote a set of ideal class representatives for $F$.  Then for $1 \le j \le h_F$, we require
\[ f_j(y \cdot (i, \ldots, i)) = \det y^{k/2} (c(i, \ldots, i)+d)^k f ( \bmx t_j^{-1} & \\ & 1 \emx
y), \quad y \in \GL_2(F_\infty)^+ \]
is a classical Hilbert modular form of weight $\bfk$.
Denote by $S_\bfk(\frakN, \psi)$ the subspace of cusp forms of $M_\bfk(\frakN, \psi)$.

For $f \in M_\bfk(\frakN, \psi)$, $x \in \A_F$ and $y \in \A_F^\times$ such that
$y_\infty \gg 0$,  we have a Fourier expansion of the form
\[ f \bmx y & x \\ & 1 \emx = c_0(y \frako) |y|^{\bfk/2} + \sum_{0 \ll \zeta \in F}
c(\zeta y \frako, f) (\zeta y_\infty)^{\bfk /2} e_F(\zeta i y_\infty) \chi_F(\zeta x), \]
where $e_F(z_1, \ldots, z_d) = \exp(2\pi i \sum z_j)$ and 
$\chi_F$ is the character of $\A_F/F$ agreeing with $e_F$ at infinity.  
The coefficients $c(\frakn, f)$ are 0 unless $\frakn$ is integral (i.e., we may
take $\zeta$ to run over totally positive elements of $y^{-1}\frako$ in the above sum) 
and $c_0(y \frako)$ is a function of strict ideal classes that is 0 
unless $\bfk$ is a parallel weight.
If each $f_j$ is a classical Hilbert modular form for the subgroup denoted
$\Gamma(t_j \mathfrak d_F, \frakN)$ in \cite{shimura}, then the adelic Fourier expansion corresponds to the classical Fourier
expansions
\[ f_j(z) = \sum a_j(\xi) e_F(\xi z), \quad \xi \in t_j \text{ such that } \xi = 0 \text{ or } \xi \gg 0, \]
where $a_j(\xi) = \xi^{\bfk/2} c(\xi t_j^{-1}, f)$ for $\xi \gg 0$ in $t_j$
and, in the case of parallel weight, $a_j(0) = N(t_j^{\bfk/2}) c_0(\eta t_j^{-1})$ 
for any $\eta \gg 0$.

For an integral ideal $\frakn$ of $F$, one defines the Hecke operator
$T_\frakn$ to be the sum over distinct $W(\frakN) y W(\frakN)$ where $y \in
Y(\frakN)$ such that $(\det y)\frako = \frakn$.  Here, if $W(\frakN) y W(\frakN) =
\bigsqcup y_j W(\frakN)$, the action on $f$ is given by
$\sum_j \psi_Y(y_j)^{-1} \pi(y_j) f$. Then $T_\frakn$ factors
as a product of local Hecke operators $T_{\frakn_\frakp}$ as in
\eqref{eq:hecke-fact}, and one readily sees the definition of the local Hecke operators
$T_{\frakn_\frakp}$ on $\GL_2(F_v)$ matches the definition on the
local quaternionic groups $B^\times_v$ when the local levels $\frakN_v = 1$
(so necessarily $B^\times_v \simeq \GL_2(F_v)$).  
Denote by $\mathcal H^S$ the Hecke algebra generated by $T_\frakn$ for
$\frakn$ coprime to $\frakN$.

Define the normalized Fourier coefficient $C(\frakn,f) = N(\frakn)c(\frakn,f)$.
(Our normalization is different from that in \cite{shimura} when $\bfk \ne (2,2,\ldots,  2)$.)
If $f$ is a common eigenfunction of the Hecke operators $T_\frakn$, 
and $f$ is normalized so that $C(\frako, f) = 1$, then the eigenvalue of $T_\frakn$
is $\lambda_f(\frakn) = C(\frakn, f)$.  If further $f$ is a cusp form,
it generates an irreducible cuspidal automorphic representation $\pi = \pi_f$
of weight $\bfk$ (each archimedean component $\pi_{\nu_j}$ is the
discrete series of weight $k_j$) with central character $\omega_\pi = \psi$ 
such that $c(\pi) = \prod_\frakp \frakp^{c(\pi_\frakp)}$ divides $\frakN$.  
We say such an $f$ is a newform if $c(\pi) = \frakN$, and denote the
span of newforms in $S_\bfk(\frakN,\psi)$ by $S_\bfk^\new(\frakN, \psi)$.

We have an Atkin--Lehner type decomposition in terms of newforms:
\[ S_\bfk(\frakN,\psi) = \bigoplus_{\frakM | \frakN} \, \bigoplus_{\, \mathfrak d |
\frakN \frakM^{-1}} \iota_{\mathfrak d} \left(S_\bfk^\new(\frakM, \psi) \right), \]
where the $\iota_{\mathfrak d}$ are the embeddings of 
$S_\bfk^\new(\frakM, \psi)$ into $S_\bfk(\mathfrak d \frakM, \psi)$ defined by
$C(\frakn, \iota_{\mathfrak d} f) = C(\frakn {\mathfrak d}^{-1}, f)$
(see \cite[Section 3]{shemanske-walling}).

The usual dictionary between modular forms and automorphic representations
defines an isomorphism (initially as vector spaces, but also as Hecke modules
with Hecke operators appropriately normalized):
\begin{equation} \label{eq:HMF-decomp}
 S_\bfk(\frakN, \psi) \simeq \bigoplus \pi_\bfk^{K_1(\frakN)},
\end{equation}
where $\pi$ runs over irreducible holomorphic weight $\bfk$ 
cuspidal automorphic representations of $\GL_2(\A)$ with central character $\psi$,
$K_1(\frakN)$ is the level $\frakN$ compact open subgroup of 
$\GL_2(\hat F)$ of ``type $\Gamma_1$,'' and $\pi_\bfk$ denotes the subspace
of $\pi$ of weight $\bfk$ vectors.

\subsection{Eisenstein series}

The Jacquet--Langlands correspondence is a only statement about 
\emph{cuspidal} representations.  We also want to know that the 
1-dimensional representations of quaternion algebras correspond to weight 
$\two=(2,2, \ldots, 2)$ Eisenstein series, which we explain here.

Let $B$, $\calO$ (a special order of level $\frakN$), $\omega$ and $\Omega$ 
be as in \cref{sec:QMFs}.  Let $\mu$ be a character of $F^\times \bs
\A^\times$ and suppose $\phi = \mu \circ N \in M_\zero(\calO, \Omega)$,
which means $\mu^2 = \omega$ and $\Omega$ agrees with $\mu \circ N$ on
$\hat \calO^\times$.  Then it is immediate from \eqref{eq:hecke-oper}
that $\phi$ is an eigenform for each Hecke operator $T_\frakn$ with
eigenvalue $\lambda_{\frakn}(\phi) = \sum \Omega^{-1}(\beta) \mu (N(\beta))$,
where $\beta \in \hat \calO(\frakn)/\calO^\times$.  By \cref{rem:41},
$\lambda_{\frakp^m}(\phi)$ is just the number of right $\calO$-ideals of norm
$\frakp^m$ if $\mu$ is unramified at $\frakp$.

Given an idele class character $\mu$ of $F$, we view this as a character on
ideals of $F$ by setting $\mu(\frakp)$ to be $\mu_\frakp(\varpi_\frakp)$
if $\mu_\frakp$ is unramified and 0 otherwise.

\begin{lemma}  Let $\fraka$, $\frakb$ be coprime integral ideals of $F$.
If $F=\Q$, assume $\fraka \ne \Z$. 

(i) There is an Eisenstein series $E_{\two,\fraka, \frakb} \in M_\two(\fraka \frakb, 1)$
 such that $C(\frakn, E_{\two,\fraka,\frakb}) = \sum_{\mathfrak d} N(\mathfrak d)$
 where $\mathfrak d$ runs over integral ideals dividing $\frakn$
 such that $\mathfrak d$ is coprime to $\fraka$ and ${\mathfrak d}^{-1} \frakn$
 is coprime to $\frakb$.  Moreover, $c_0(y \frako, E_{\two,\fraka,\frakb}) = 2^{-d} \zeta_F(-1)
 \prod_{\frakp | \fraka \frakb} (1 - N(\frakp)^{-1})$ for $y \in \A_F^\times$ with $y_\infty \gg 0$.
 
 (ii) Let $\mu$ be a nontrivial finite-order idele class character of $F$, with conductor $\frakc | \fraka \frakb$.  The twisted Eisenstein series
$E_{\two,\fraka,\frakb}(\mu) := E_{\two, \fraka,\frakb} \otimes \mu \in 
M_\two(\fraka \frakb \frakc, \mu^2)$
satisfies $C(\frakn, E_{\two,\fraka,\frakb}(\mu)) = \mu(\frakn) C(\frakn, E_{\two,\fraka,\frakb})$
and $c_0(y \frako, E_{\two,\fraka,\frakb}(\mu)) = 0$.
 \end{lemma}
 
 \begin{proof} Let $\mu_\frakM$ denote the character of ideals
 obtained by viewing $\mu$ as an (imprimitive if $\mathfrak c \ne \frakM$)
 character of conductor $\frakM$ (i.e., it is 0 on integral ideals not prime to 
 $\frakM$).
 The Eisenstein series $E_{\two,\fraka,\frakb}(\mu)$ is the Eisenstein
 series associated to the characters on ideal groups 
 $\eta=\mu_{\fraka}$ and $\chi = \mu_{\frakb}$ in \cite[Proposition 3.4]{shimura}, 
 which gives the level, central character and
 a formula for $C(\frakn, E_{2,\fraka,\frakb}(\mu))$ as in (i).
 (Technically, \cite[Proposition 3.4]{shimura} assumes $F\ne \Q$, but the elliptic case 
 is classical.)
 The constant term is calculated in \cite[Proposition 2.1]{DDP}.
 The explicit relation for twists of Fourier coefficients in (ii) is given in 
 \cite[Proposition 9.7]{shimura:1980}.  
 \end{proof} 
 
 Note that the Fourier coefficients of $E_{\two,\fraka,\frakb}(\mu)$
are multiplicative, which means that $E_{\two,\fraka,\frakb}$
is an eigenfunction at least for the unramified Hecke algebra.
We also note that $E_{\two,\fraka,\frakb}$ only depends on the squarefree
parts $\fraka_0$ and $\frakb_0$ of $\fraka$ and $\frakb$, so in fact
we have $E_{\two,\fraka,\frakb}(\mu) \in M_\two(\fraka_0 \frakb_0 \frakc,
\mu^2)$, where $\frakc$ is the conductor of $\mu$.
 
\begin{prop} \label{prop:Eis-corr}
Let $\mu$ be a finite-order idele class character of $F$
such that $\phi_\mu = \mu \circ N \in M_\zero(\calO,\Omega)$.
Write $\frakN = \frakN' \frakM$ where $\frakN'$ (resp.\ $\frakM$)
is of the form $\prod \frakp^{r_\frakp}$ where $\frakp$ runs over all finite
primes at which $B$ is ramified (resp.\ split).
Then $\phi_\mu$ is an eigenform whose Hecke
eigenvalue $\lambda_\frakn(\phi)$ for $T_\frakn$ is 
$C(\frakn,E_{\two,\frakN',\frakM}(\mu))$ for all $\frakn \nmid \frakN$.
Moreover, $E_{\two,\frakN',\frakM}(\mu) \in M_\two(\frakN, \mu^2)$.
\end{prop}

\begin{proof}
Since the Hecke eigenvalues are multiplicative, 
for the first part it suffices to consider prime power eigenvalues 
$\lambda_{\frakp^m}(\phi_\mu)$.
Suppose $\frakp \nmid \frakN$.  Then $\Omega_\frakp = 1$ and
$\mu_\frakp = 1$.  By \cref{rem:41}, $\lambda_{\frakp^m}(\phi_\mu)$
is $\mu(\frakp^m)$ times the number of $\calO_\frakp$-ideals of norm
$\frakp^m$.  It is well known that this is $\frac{q_\frakp^{m+1}-1}
{q_\frakp -1} = C(\frakp^m,E_{\two,\frakN',\frakM})$.  This proves
the first assertion.

To see $E_{\two,\frakN',\frakM}(\mu) \in M_\two(\frakN, \mu^2)$,
recall $c(\mu_\frakp \circ N) = \max \{ 2c(\mu_\frakp), 1 \}$ for $\frakp | \frakN'$
and, by the proof of \cref{lem:eis-pM}, $c(\mu_\frakp) = 0$ for $\frakp | \frakM$.
Since $c(\mu \circ N) | \frakN$, we see $c(\mu_\frakp) = 0$
when $v_\frakp(\frakN) \le 1$ and $c(\mu_\frakp) \le \frac{v_\frakp(\frakN)}2$
otherwise.  Hence $c(\mu) \frakN_0 | \frakN$, where $\frakN_0$ denotes
the squarefree part of $\frakN$.
\end{proof}

\begin{rem}
Since the above proposition only considers Hecke eigenvalues 
away from the level, there are other Eisenstein series we could have used
as well.  However, we chose $E_{\two,\frakN',\frakM}$
to correspond to the constant function $\phi_1 = 1$ on $B^\times(\A)$ because, at least in the case $\frakN$ is squarefree, the appropriate
definition of ramified Hecke operators makes the ramified
Hecke eigenvalues of $\phi_1$ match with the ramified Hecke eigenvalues
of $E_{\two,\frakN',\frakM}$.  Here, by appropriate definition of ramified
Hecke operators, we mean that one should use the same definition
of local ramified Hecke operators for $\frakp | \frakM$ on 
$M_\zero(\calO, \Omega, \omega)$ as on $M_\two(\frakN, \omega)$.
See \cite{me:cong} or \cite{me:cong2} for $\frakp | \frakN'$.

We also note that for a suitably normalized inner product $(\phi_1, \phi_1)$
is the mass $m(\calO)$ of $\calO$, and $(\phi_\mu, \phi_\mu) = 0$ for $\mu \ne 0$.
These inner products correspond to the constant terms 
$c_0(y\frako, E_{\two, \frakN', \frakM}(\mu))$,
up to simple factors.  For instance, if $\calO$ is maximal, then
$m(\calO) = 2h_F N(\frakN) |c_0(\frako, E_{\two, \frakN', \frakM}(\mu))|$
(see \cite[(1.6)]{me:cong}).
\end{rem}

\subsection{Correspondence of Hecke modules} \label{sec:hecke-corr}

Here we come to the main results of this section.

We summarize our notation from above: 
$B$ is a totally definite quaternion algebra
over $F$ with discriminant $\frakD$; $\frakN$ is a nonzero integral ideal in 
$\frako=\frako_F$ such that $\frakp | \frakN$ for all finite $\frakp$ where 
$B_\frakp$ is division; $\calO$ is a special order in $B$ of level 
$\frakN = \frakN_1 \frakN_2 \frakM = \prod \frakp^{r_\frakp}$ as in
\cref{sec:quat-nfs}; 
and $\Omega$ is a semigroup homomorphism
of $\hat \calO^\bullet$ extending $\omega|_{\hat \frako^\times}$ 
as in \cref{sec:qmfs}.  
Recall also
that $c(\omega_\frakp) \le \frac{r_\frakp}2$ for all $\frakp | \frakN_1 \frakN_2$.
We will further assume $\calO$ is chosen so that $r_\frakp$ is even for all
$\frakp | \frakN_2$; i.e., we choose $\calO$ so that it is of unramified
quadratic type at as many places as possible (given $\frakN$).

We call an eigenform \emph{$\frakp$-primitive}
if the associated local representation $\pi_\frakp$ is minimal. 
In classical language, this means that there is no $\frakp$-power conductor character $\chi$ such that twisting by $\chi$ lowers the level at $\frakp$.
For an ideal $\fraka$ in $\frako$, we say a form is $\fraka$-primitive if it is 
$\frakp$-primitive for all $\frakp | \fraka$.

The following is the first, unrefined version of our 
``classical'' Jacquet--Langlands correspondence.

\begin{prop} \label{prop:JL-prim}
There is a (non-canonical) homomorphism of $\mathcal H^S$-modules
\[ \JL: S_\bfk^\new(\calO, \Omega, \omega) \to S_{\bfk+\two}^\new(\frakN, \omega) \]
such that any $\frakD$-primitive newform $f \in S_{\bfk+\two}^\new(\frakN, \omega)$
lies in the image of this map.
\end{prop}

\begin{proof} Consider any $\pi$ appearing in the decomposition of 
$S_\bfk^\new(\calO, \Omega, \omega)$
from \cref{prop:quat-AL-decomp}(i).  By the same proposition,
the global Jacquet--Langlands correspondence associates to $\pi$ an irreducible
cuspidal automorphic representation $\pi'$ of $\GL_2(\A)$ with central character
$\omega$ such that $\pi'_v \simeq \pi_v$ for all $v \nmid \frakD$ 
and $c(\pi') := \prod \frakp^{c(\pi'_\frakp)} = \frakN$.  Moreover, for each $\frakp | \frakD$,
$\pi'_\frakp$ is a discrete series representation, i.e., special or supercuspidal, and
$\pi'_{\nu_i}$ is the discrete series of weight $k_{i}+2$ for each infinite place
$\nu_i$.  Further,
all $\pi'$ with central character $\omega$, conductor $\frakN$ and holomorphic 
weight $\bfk+\two$ such that $\pi'_\frakp$ is discrete series for $\frakp | \frakD$ appear in the
image of the representation-theoretic Jacquet--Langlands correspondence.
Such $\pi'$ will correspond to a $\pi$ appearing in the decomposition 
\eqref{eq:Sk-decomp} if $\pi_\frakp^{\Omega_\frakp} \ne 0$ for $\frakp | \frakD$.

Since, for $\frakp | \frakD$, $\pi'_\frakp$ being minimal implies $\pi'_\frakp$
is discrete series (in fact supercuspidal if $c(\pi'_\frakp) > 1$) whenever $c(\pi'_\frakp) > 0$, \cref{thm:34} and \cref{thm:35}
imply that all $\frakD$-primitive representations appearing in the spectral decomposition
of $S_\bfk^\new(\frakN, \omega)$, i.e., the new part
 of the decomposition \eqref{eq:HMF-decomp}, lie in the image of this correspondence.

Consequently, we can define a map at the level of modular forms as follows.
For each $\pi$ as above, fix a basis $\phi_1, \ldots, \phi_m$ of $\pi_\bfk^\Omega$.  
Map each $\phi_i$ to the unique (normalized) newform $f_{\pi'} \in 
 S_\bfk^\new(\frakN, \omega)$ associated to $\pi'$.  Extending this by linearity
gives an $\mathcal H^S$-module homomorphism as $\pi'_v \simeq \pi_v$
for $v \nmid \frakD$.
\end{proof}

\begin{rem} 
In fact, it follows from \cref{prop:hecke-comp1} and \cref{prop:hecke-comp2}
that the Hecke eigenvalues for all 
$T_\frakp$'s are also preserved under the above correspondence of newforms.
See also \cite{me:cong2} for the case of squarefree conductor, where this was used to
produce congruence of eigenforms mod 2.
This preservation of eigenvalues will no longer be true for $T_\frakp$'s with
$\frakp | \frakM$ when we extend the correspondence to
include oldforms.
\end{rem}

The above result is sufficient to tell us that the basis problem has a solution, but we want
to know more precise information about this map, namely what can we say about its
kernel and its image, as well as understanding how it can be extended to
$M_\bfk(\calO,\Omega, \omega)$.
It is clear from the above proof that understanding the kernel amounts to
understanding $\dim \pi_\bfk^\Omega$, which was the main goal of our local calculations.

To get a more precise description of this map, and its extension
to $S_\bfk(\calO, \Omega, \omega)$, it will be convenient to define certain
refinements of $S_\bfk(\frakN, \psi)$ and $S^\new_\bfk(\frakN, \psi)$.  

Let $\fraka$, $\frakb$, $\frakc$ and $\frakd$ be nonzero pairwise
coprime ideals in $\frako$
dividing $\frakN$ such that $\frakp^2 | \frakb$ whenever $\frakp | \frakb$.  
We define the subspace 
$S^{[\fraka; \frakb, \frakc; \frakd]}_\bfk(\frakN, \psi)$ of $S_\bfk(\frakN, \psi)$ by
\[ S^{[\fraka; \frakb, \frakc; \frakd]}_\bfk(\frakN, \psi) :\simeq \bigoplus \pi^{K_1(\frakN)}, \]
where $\pi$ runs over representations as in \eqref{eq:HMF-decomp} satisfying:
(i) $c(\pi_\frakp) = v_\frakp(\frakN)$ for $\frakp | \fraka \frakb \frakc \frakd$;
(ii) $\pi_\frakp$ is discrete series for $\frakp | \fraka$;
(iii) $\pi_\frakp$ is minimal supercuspidal for $\frakp | \frakb$; and
(iv) $\pi_\frakp$ is special for $\frakp | \frakc$.
If $f \in S_\bfk(\frakN, \psi)$ is the newform (not necessarily of level $\frakN$) 
associated to $\pi$, (i) means that $f$
is $\frakp$-new for each $\frakp | \fraka \frakb \frakc \frakd$, i.e., $f$ is in the
orthogonal complement of forms coming from level $\frakp^{-1} \frakN$ for
such $\frakp$;
(ii) means that the minimum $\frakp$-part of the level among 
$\frakp$-power twists $f \otimes \chi$ of $f$ is strictly greater than 
the $\frakp$-power conductor of $\psi \chi^2$; 
(iii) means that the $\frakp$-power of level
of $f$ is minimal among twists, at least $2$, and this power is strictly greater than (in fact
at least twice) the $\frakp$-power conductor of the nebentypus; and
(iv) means that one can twist $f$ such that the $\frakp$-part of
the level is $\frakp$ and the nebentypus is prime to $\frakp$.
 In particular, if the nebentypus conductor is prime to $\frakb \frakc$, 
 (iii) just means $f$ is $\frakb$-primitive and (iv) just means $\frakp$ sharply
 divides the exact level of some twist of $f$.  
  We also write 
$S^{[\fraka; \frakb; \frakd]}_\bfk(\frakN, \psi)$ for 
$S^{[\fraka; \frakb, \frako; \frakd]}_\bfk(\frakN, \psi)$
and $S^{\frakd{\text -}\new}_\bfk(\frakN, \psi)$ for
$S^{[\frako, \frako, \frako, \frakd]}_\bfk(\frakN, \psi)$.

All isomorphisms below are as $\mathcal H^S$-modules.

\begin{thm} \label{JL-thm-new}

(i) As $\mathcal H^S$-modules,
\[ S^\new_\bfk(\calO, \Omega, \omega) \simeq
2^{\# \{ \frakp | \frakN_2 \} }
S^{[\frakN_1; \frakN_2; \frakM]}_{\bfk+\two}(\frakN, \omega) \oplus 
R^\new_\bfk(\calO, \Omega, \omega), \]
where the ``remainder'' space $R_\bfk^\new(\calO, \Omega, \omega)$ is the 
subspace of $S^\new_\bfk(\calO, \Omega, \omega)$
given by $R^\new_\bfk(\calO, \Omega, \omega) = 
\bigoplus \pi_\bfk^\Omega$ where $\pi$ runs over representations as in
\eqref{eq:Sk-decomp} such that $\pi_\frakp$ is non-minimal for some
$\frakp | \frakN_2$.  

(ii) If $\omega_\frakp$ is unramified for each $\frakp | \frakD$, and each
$\frakp | \frakN_2$ such that $v_\frakp(\frakN_2) \ge 4$ is non-dyadic, then
\[ S^\new_\bfk(\calO, \Omega, \omega) \simeq
\bigoplus 2^{\# \{ \frakp | \frakN_2' \} }
S^{[\frakN_1; \frakN_2', \frakN_2''; \frakM]}_{\bfk+\two}(\frakN, \omega), \]
where $(\frakN_2', \frakN_2'')$ runs over all pairs of coprime divisors of 
$\frakN_2$ such that $\frakN_2' \frakN_2'' = \frakN_2$ and $v_\frakp(\frakN_2'')
=2$ for all $\frakp | \frakN_2''$.
\end{thm}

Note $S^{[\frakN_1; \frakN_2; \frakM]}_{\bfk+\two}(\frakN, \omega)$
is just the $\frakN_1 \frakN_2$-primitive subspace of 
$S^{\new}_{\bfk+\two}(\frakN, \omega)$ by our assumptions on $\frakN_1$
and $\frakN_2$.

\begin{proof}
Consider an arbitrary $\pi$ appearing in the decomposition of
$S_\bfk^\new(\calO, \Omega, \omega)$ from \cref{prop:quat-AL-decomp}(i),
and let $\pi'$ be its Jacquet--Langlands transfer to $\GL_2(\A)$.

First consider $\frakp | \frakN_1$.  If $\frakp \| \frakN_1$, then
$\omega_\frakp$ is unramified and $\dim \pi_\frakp = 1$, i.e., $\pi_\frakp'$
is an unramified twist of Steinberg.  If $\frakp^3 | \frakN_1$, then
$\dim \pi_\frakp > 1$ and as $\pi_\frakp$ has odd conductor it must be 
minimal.  Hence $\pi_\frakp'$ is a minimal supercuspidal.
In either case, by \cref{thm:34}, we have
$\dim \pi_\frakp^{\Omega_\frakp} = 1$.

Now suppose $\frakp | \frakN_2$.  Since $\frakp^2 | \frakN_2$, 
$\dim \pi_\frakp = 1$ implies $\pi_\frakp$ is non-minimal.
So assume $\dim \pi_\frakp > 1$.
By \cref{thm:35}, if $\pi_\frakp$
is minimal, then $\dim \pi_\frakp^{\Omega_\frakp} = 2$ and $\pi'_\frakp$
is a minimal supercuspidal.  Recall that,
by \cref{prop:loc-non-min}, $\pi_\frakp$ is necessarily minimal if 
$\frakp$ is non-dyadic and $c(\omega_\frakp) < \frac{v_\frakp(\frakN_2)}2$.
On the other hand, if $c(\pi_\frakp) = 2$ ($\frakp$ dyadic or not),
$\pi_\frakp$ is necessarily minimal.

From the characterization of the image of the Jacquet--Langlands
correspondence, we get all $\pi'$'s appearing in the new part of
\eqref{eq:HMF-decomp} such that $\pi'_\frakp$ is of the type
specified above for $\frakp | \frakN_1 \frakN_2$.  This proves (i).

Under the assumptions of (ii), the above argument implies the only contribution to 
$R_\bfk^\new(\calO, \Omega, \omega)$ occurs for $\pi$ such that
$\dim \pi_\frakp = 1$ for some $\frakp | \frakN_2$.  Consider such
$\pi$ and $\frakp$.  By our assumption,
$\Omega_\frakp = 1$, which by \cref{lem:1-d} implies $c(\pi_\frakp) =2$,
i.e., $\pi_\frakp$ corresponds to a ramified quadratic twist of Steinberg.
From the above calculations, 
one sees $\dim \pi_\bfk^\Omega = 2^{\# \{ \frakp | \frakN_2 :
\dim \pi_\frakp > 1\}}$. 
\end{proof}

\begin{rem} \label{rem:2-powers}
One can also extend (ii) to allow for
dyadic primes dividing $\frakN_2$ to a
sufficiently large power, using \cref{prop:loc-non-min}(iii).
For instance if $F=\Q$, we can allow $\frakN_2$ to be even if $2^8 | \frakN_2$,
or just $2^6 | \frakN_2$ if we take $\calO_2 \simeq \calO_{r_2}(E_2)$ where
$E_2$ is $\Q_2(\sqrt 3)$ or $\Q_2(\sqrt 7)$.  We note the same issue
arises in \cite{HPS:mams}, e.g., see Theorem 7.30 and Example 10.7 
of \emph{op.\ cit.}
\end{rem}

It is also easy to see that if we replace 
$S^\new_\bfk(\calO, \Omega, \omega)$ with the $\frakN_1 \frakN_2$-new
subspace, we get the analogous result by working with
$\frakN_1 \frakN_2$-new spaces on the Hilbert modular form side.
Via the decomposition from \cref{cor:quat-AL-decomp2} or 
\eqref{eq:AL-decomp2}
of $S_\bfk(\calO, \Omega, \omega)$ into newspaces of smaller levels,
we can use this to describe the Hecke module structure of 
$S_\bfk(\calO, \Omega, \omega)$ in terms of spaces of Hilbert modular
forms, at least in the case that $\frakN_2$ is cube-free
(recall we have not described $S^\new_\bfk(\calO, \Omega, \omega)$
when $v_\frakp(\frakN_2)$ is odd for some $\frakp | \frakN_2$).
This can be regarded as a classical Jacquet--Langlands 
correspondence for the full cuspidal space.  We just write things down
in the simplest case of nebentypus conductor prime to $\frakD$,
so that the ``remainder'' spaces do not appear.

\begin{cor} \label{cor:cusp-JL}
Suppose $\frakN_2$ is cube-free, and $\omega_\frakp$ is unramified for each
$\frakp | \frakD$.  Then
\[ S_\bfk(\calO,\Omega, \omega) \simeq \bigoplus 2^{\# \{ \frakp | \frak b \} }
S_{\bfk + \two}^{[\fraka; \frakb, \frakc; \frako]}(\fraka \frakb \frakc \frakM, \omega), \]
where $\fraka$, $\frakb$, $\frakc$ run over divisors of $\frakN_1 \frakN_2$ such that (i) $\frakD | \fraka \frakb \frakc$;
(ii) $\fraka | \frakN_1$ and $v_\frakp(\fraka)$ is odd for $\frakp | \frakN_1$;
(iii) $\frakb \frakc | \frakN_2$ with $\frakb$, $\frakc$ coprime; and (iv)
$\frakb$ is a square.

In particular, if $\calO$ is of unramified quadratic type, i.e.\
$\frakN_2 = 1$, we have
\[ S_\bfk(\calO,\Omega, \omega) \simeq
\bigoplus_{\frakd} S_{\bfk+\two}^{\frakd {\text -} \new}(\frakd \frakM, \omega), \]
where $\frakd$ runs over all divisors of $\frakN_1$ such that
$v_\frakp(\frakd)$ is odd for all $\frakp | \frakN_1$.
In this case,
\[ S_\bfk^\new(\calO,\Omega, \omega) \simeq
S_{\bfk+\two}^{\frakN_1{\text -}\new}(\frakN_1 \frakM, \omega). \]
\end{cor}

\begin{proof}
From \eqref{eq:AL-decomp2}, we deduce that
$S_\bfk(\calO,\Omega, \omega) \simeq \bigoplus S_\bfk^{\fraka \frakd{\text -}\new}(\calO'(\fraka \frakd \frakM), \Omega', \omega)$,
where $\fraka$ runs over divisors of $\frakN_1$ satisfying (ii), $\frakd$ runs over divisors of $\frakN_2$ such that $\frakD | \fraka \frakd$,
and $\calO'(\frakN')$ denotes a special order in $B$ of level $\frakN'$
with an admissible extension $\Omega'$ of $\Omega$.

For fixed $\fraka$, $\frakd$, we can decompose
\[ S_\bfk^{\fraka \frakd{\text -}\new}(\calO'(\fraka \frakd \frakM), \Omega', \omega) = \bigoplus S_\bfk^{[\fraka; \frakb, \frakc; \frako]}(\calO'(\fraka \frakd \frakM), \Omega', \omega), \]
where $\frakb, \frakc$ run over relatively prime divisors of $\frakd$
such that $\frakb \frakc = \frakd$, and each space on the right
denotes the $\fraka$-new subspace of $S_\bfk(\calO'(\fraka \frakd \frakM), \Omega', \omega)$ consisting of $\pi$ such that $\pi_\frakp$ is
1-dimensional for $\frakp | \frakc$ and higher dimensional for 
$\frakp | \frakb$.  Since $\frakb$ is cube-free, this means that
$\pi_\frakp$ is also minimal for $\frakp | \frakb$.  Also note this
space is can only be nonzero if $\frakb$ is a square.

The (proof of the) above theorem gives 
$S_\bfk^{[\fraka; \frakb, \frakc; \frako]}(\calO'(\fraka \frakd \frakM), \Omega', \omega)
\simeq 2^{\# \{ \frakp | \frakb \} }
S_{\bfk + \two}^{[\fraka; \frakb, \frakc; \frako]}(\fraka \frakb \frakc \frakM, \omega)$.
\end{proof}

Finally, when $\bfk=\zero$, we also want to describe the full
space of quaternionic modular forms.  This is desired, for instance,
to construct Eisenstein congruences---see \cite{me:cong}, \cite{me:cong2}.

\begin{prop} \label{prop:eis-JL}
We have 
\[ E_\zero(\calO,\Omega,\omega) \simeq
\bigoplus_\mu \C E_{\two,\frakN_1 \frakN_2,\frakM}(\mu), \]
where $\mu$ runs over characters of $F^\times \bs \A_F^\times$ such that
$\mu^2 = \omega$, $c(\mu)^2 | \frakN_1 \frakN_2$, and
$\mu_\frakp \circ N$ agrees with $\Omega_\frakp$ on $\frako_{E_\frakp}$
 for all $\frakp | \frakD$, 
where $E_\frakp/F_\frakp$ denotes a quadratic extension such that $\calO_\frakp
\simeq \calO_{r_\frakp}(E_\frakp)$.

In the case that $\calO$ is of unramified quadratic type and $\omega$ is unramified, then $\mu$ simply runs over all unramified characters of 
$F^\times \bs \A_F^\times$ such that $\mu^2 = \omega$.  If in addition 
$\omega = 1$ and $h_F$ is odd,
then $E_\zero(\calO,\Omega,\omega) \simeq \C E_{\two,\frakN',\frakM}$.
\end{prop}

\begin{proof} This is clear from \cref{lem:1-d}, \cref{lem:eis-pM}
and \cref{prop:Eis-corr}.
\end{proof}

\subsection{A congruence application} \label{sec:cong}

One application of \cref{cor:cusp-JL} is that it can be used to refine
Eisenstein congruence results from \cite{me:cong}.  To us, the main deficiency
in the results from \cite{me:cong} is that we could not show we get
Eisenstein congruences with newforms when we work with non-maximal
orders in the relevant quaternion algebra.  At least in some situations,
\cref{cor:cusp-JL} can be used to address this, but we only discuss
a very simple case for elliptic modular forms here.  

Denote by $E_{2,p}$ the normalized
Eisenstein series in $M_2(p)$, and $E_{2,p^2}(z) =
E_{2, p}(z) - E_{2,p}(pz)$.  
The $n$-th Fourier coefficient of $f \in M_2(N)$
is denoted $a_n(f)$.

\begin{prop} Let $p \ge 3$.  Then there exists a newform $f \in S_2^\new(p^3)$
such that $a_n(f) \equiv a_n(E_{2,p^2}) \mod p$ for all $n$.
\end{prop}

\begin{proof} 
Since $a_n(E_{2,p^2}) = 0$ when $p | n$, and this is also true for 
any $a_n(f)$ for a newform $f \in S^\new_2(p^3)$, it suffices
to prove the above congruence for $n$ prime to $p$.

In \cite[Corollary 2]{me:cong}, we proved the existence of an 
 eigenform $f \in S_2(p^3)$ (not necessarily new) 
 satisfying the above congruence for $n$ prime to $p$.  
The proof comes via constructing a quaternionic eigenform
$\phi \in S_0(\calO)$ Hecke congruent mod $p$ 
to the quaternionic Eisenstein series
$\phi_0 = 1 \in E_0(\calO)$, where $\calO$ is a special order of level $p^3$
in the quaternion algebra $B/\Q$ with discriminant $p$, and applying the
Jacquet--Langlands correspondence to transfer $\phi$ to $f$.

By \cref{cor:cusp-JL}, this means our $f$ as above in fact lies in
$S_2(p) \oplus S_2^\new(p^3)$.  (One can also derive this special case
from the results in \cite{pizer:algorithm}.)  Hence it will suffice to show
there is no $f \in S_2(p)$ satisfying this congruence.  This follows
from Mazur's determination of Eisentein ideals for $S_2(p)$ 
\cite[Proposition 9.7]{mazur}.
\end{proof}

We expect that one can use \cref{cor:cusp-JL} to refine more general Eisenstein 
congruence results in \cite{me:cong} by analyzing the behavior of ideal
classes upon passing to suborders (sidestepping the use of Mazur's
result).  We hope to address this elsewhere.

%
%

\section{Theta series} \label{sec:theta}

Here we explain how to reinterpret the ``classical'' Jacquet--Langlands map of the
previous section in the more historically classical context of theta series.  This will
extend  the solutions to the basis problem given in \cite{eichler} and \cite{HPS:mams}
for $F=\Q$ to both more general quaternion algebras and to totally real $F$.
We keep notation and assumptions as in \cref{sec:hecke-corr}.

For totally real fields, Eichler  studied theta series attached to Brandt matrices (without
character) of maximal orders \cite{eichler:1977}, though did not solve the basis problem in 
this setting.  Shimizu \cite{shimizu} effectively gave a representation-theoretic solution to the
basis problem in this setting when each $k_v > 2$.  Shimizu's solution was in terms of certain
adelic theta series, but he did not explicate how to reinterpret these as classical theta series.  
This issue of realizing Shimizu's theta series as computable, classical theta series was 
taken up in the thesis \cite{gebhardt}.  The theta series here, like Eichler's, 
are given in terms of Brandt  matrices, and thus computable 
(e.g., see \cite{DV} for how to compute Brandt  matrices for Eichler orders),
and provide a solution to the basis problem for Hilbert modular forms.

We define the Brandt matrix series to be the matrix of functions of $x \in \A$, 
$y \in \A^\times$ with $y_\infty \gg 0$ given by
\begin{equation} \label{eq:theta}
 \Theta \bmx y & x \\ & 1 \emx = A_0  |y|^{\bfk/2 + \one} +
\sum_{0 \ll \zeta \in F} A_{\zeta y \frako} \cdot N(\zeta y \frako)^{-1}  (\zeta y_\infty)^{\bfk /2 + \one} e_F(\zeta i y_\infty)
\chi_F(\zeta x),
\end{equation}
where the Brandt matrices $A_\frakn$ are as in \cref{sec:brandt} for $\frakn$ integral 
and interpreting 
$A_\frakn = 0$ if $\frakn$ is not integral.  This corresponds
to the collection of classical matrix Fourier series given by
\[ \Theta^{(m)} (z) = N(t_m) A_0 + N(t_m) \sum_{0 \ll \xi \in t_m} A_{\xi t_m^{-1}} \xi^{\bfk/2} e_F(\xi z), \quad 1 \le m  \le h_F. \]
As in \cref{sec:HMFs}, $t_1, \dots, t_{h_F}$ are ideal class representatives for $F$.
(To rewrite the constant term, we used that $A_0 = 0$ unless $\bfk = \zero$.)

Consider some entry $\theta(z) = \sum a(\xi) e_F(\xi z)$ in the $(i,j)$-th 
$\kappa \times \kappa$ block of $N(t_m)^{-1} \Theta^{(m)}(z)$.  Then, for $0 \ll \xi \in t_m$, 
we can write
\[ a(\xi) = \frac 1{e_j}\sum_\gamma \Omega^{-1}(x_i^{-1} \gamma x_j) r(\gamma) \xi^{\bfk/2}, \]
where $\gamma$ ranges as in \cref{prop:brandt} for $\frakn = \xi t_m^{-1}$
and $r$ is some matrix coefficient of $\rho_\bfk$.  We also have
$a(0) = 0$ unless $\bfk = \zero$, $i=j$ and $\Omega=1$, in which case $a(0) = 
\frac 1{e_i}$.

Note $\theta$ can only be nonzero if $N(\calI_i \calI_j^{-1})$ lies in the same ideal class 
as $t_m$, so assume this. 
The coefficients of $\Sym^\bfk$ are homogenous polynomials of
degree $\bfk$ which are ``harmonic with respect to $N_{B/F}$''.
This means the following. For each infinite place $\nu_i$, 
let $1, i, j, k$ be the standard basis for $B_{\nu_i} \simeq \mathbb H$.  Then we
can write any element $\gamma$ of $B$ as $\gamma = x + \epsilon_1 y i + \epsilon_2 z j + \epsilon_1 \epsilon_2
wk$ where $x, y, z, w \in F$ and $\epsilon_1, \epsilon_2$ lie in at most quadratic extensions
of $F$.  Then the coefficients of $\Sym^{k_i}(\gamma)$ are spherical harmonic polynomials
in $x, y, z, w$ of degree $k_i$ (where now harmonic means in the usual sense,
i.e., killed by the usual Laplacian).  See, e.g., \cite[Proposition II.6]{eichler}.

Consequently, $\theta(z)$ is a classical theta series of the type 
defined by Eichler \cite{eichler:1977} in the case of maximal orders and trivial character, 
where he proved that his theta series are classical Hilbert modular forms of suitable 
level by verifying the appropriate transformation laws.
When $F=\Q$, theta series with character were treated in 
\cite{eichler} for orders of squarefree level and in \cite{HPS:mams} for special orders.  
However, for a special order of level $N$,
\cite{HPS:mams} require some technical conditions to guarantee that 
their theta series are elliptic modular forms of level $N$---in general
they only show their theta series have level $N^2$.

It should be possible to extend these approaches for general $F$ to handle special
orders and nontrivial character, but we will take a representation-theoretic approach to
verifying our theta series are modular forms and have the desired level.  However, due
to the difference of the definitions of local ramified Hecke operators for $B^\times(\A)$ 
and $\GL_2(\A)$ (even at split places), we will only prove this for the ``new'' 
cuspidal subspace
of theta series, which is sufficient for our solution to the basis problem.

Let $\Theta_{\bfk+\two}(\calO, \Omega, \omega)$ be the subspace of
(adelic) theta series
generated by the entries of $\Theta$ which transform under the center of
$\GL_2(\A)$ by $\omega$.  Algorithmically this subspace can be described
as follows.  By \eqref{eq:brandt}, we can simultaneously block diagonalize 
the Brandt matrices $A_\frakn$ so that each block is either zero or acts as the restriction of $T_\frakn$ to $M_\bfk(\calO, \Omega, \omega)$.  Call a block of
the latter type $A_{\frakn, \omega}$.  This block diagonalizes
$\Theta$, giving us block matrices of functions $\Theta_\omega$, which 
can be written as in  \eqref{eq:theta} but with $A_{\frakn, \omega}$'s in 
place of the $A_\frakn$'s.
Then $\Theta_{\bfk + \two}(\calO, \Omega, \omega)$
is simply the linear span of the entries of $\Theta_\omega$.

We can similarly define the new cuspidal subspace as follows.  
We can further block
diagonalize each $A_{\frakn, \omega}$ (again, simultaneously in $\frakn$)
into three blocks $A_{\frakn,\omega}^\new$, $A_{\frakn, \omega}^\old$,
and $A_{\frakn,\omega}^{\textrm{eis}}$
which act by $T_\frakn$ on $S_\bfk^\new(\calO, \Omega, \omega)$,
$S_\bfk^\old(\calO, \Omega, \omega)$ and $E_\bfk(\calO, \Omega, \omega)$ respectively.  This yields a decomposition
of $\Theta_\omega$ into three blocks $\Theta_\omega^\new$, $\Theta_\omega^\old$ and $\Theta_\omega^{\textrm{eis}}$.
Define $\Theta_{\bfk + \two}^\new(\calO, \Omega, \omega)$ to be the linear span of the entries of $\Theta_\omega^\new$.

We define an $\mathcal H^S$-action on $\Theta_{\bfk + \two}(\calO, \Omega, \omega)$ as follows.
Consider the decomposition \eqref{eq:Mk-decomp}.  We take
a basis $\Phi$ of $\C^{h \kappa}$ consisting of $\dim \pi_\bfk^\Omega$ forms
for each $\pi$ in this decomposition and vectors spanning $\ker \Xi$,
with $\Xi$ as in \cref{sec:brandt}.  Block diagonalizing with
respect to $\Phi$ gives a decomposition
$\Theta_{\bfk+\two}(\calO, \Omega, \omega) = \bigoplus \Theta_\pi$ 
where $\Theta_\pi$ is generated by at most $\dim \pi_\bfk^\Omega$ 
theta series whose normalized Fourier coefficients are the Hecke eigenvalues
of any nonzero $\phi \in \pi$ away from $\frakN$.  For $\frakn$ coprime to $\frakN$,
we let $T_\frakn$ act on $\Theta_\pi$ by the corresponding unramified Hecke eigenvalue
for $\pi$.  Extend this action linearly to $\Theta_{\bfk+\two}(\calO, \Omega, \omega)$.

\begin{prop} \label{qmf-to-theta}
There are $\mathcal H^S$-module epimorphisms
$M_\bfk(\calO, \Omega, \omega) \to \Theta_{\bfk+\two}(\calO, \Omega, \omega)$
and
$S_\bfk^\new(\calO, \Omega, \omega) \to \Theta_{\bfk+\two}^\new(\calO, \Omega,
\omega)$.  The latter map is an isomorphism if $\calO$ is of unramified
quadratic type.
\end{prop}

\begin{proof} The first statement is clear from the above decomposition
 $\Theta_{\bfk+\two}(\calO, \Omega, \omega) = \bigoplus \Theta_\pi$. 
 The second statement follows as $\dim \Theta_\pi = \dim \pi_\bfk^\Omega = 1$
 for each such $\pi$ when $\calO$ is of unramified quadratic type.
\end{proof}

\begin{thm} \label{theta-thm1}
The full theta space $\Theta_{\bfk+\two}(\calO, \Omega, \omega)$
embeds as an $\mathcal H^S$-submodule of $M_{\bfk+\two}(\frakN, \omega)$.
In particular, for every $\theta \in \Theta_{\bfk+\two}(\calO, \Omega, \omega)$,
there exists $f \in M_{\bfk+\two}(\frakN, \omega)$ such that
their (nonzero) Fourier coefficients agree away from $\frakN$.
\end{thm}

\begin{proof}
Consider the decomposition $\Theta_{\bfk+\two}(\calO, \Omega, \omega) = \bigoplus \Theta_\pi$
as above.  For $\pi$ appearing in this decomposition, 
there exists an eigenform $f_\pi \in M_{\bfk + \two}(\frakN, \omega)$ whose
Hecke eigenvalues agree with those of any nonzero $\phi \in \pi$ outside of
$\frakN$.  This follows from \cref{JL-thm-new}(i) when $\dim \pi > 1$
and \cref{prop:eis-JL} when $\dim \pi = 1$.
Let $\frakN'$ be the exact level of $f_\pi$, i.e., the conductor of the
associated automorphic representation of $\GL_2(\A)$.
The number  of linearly independent $f_\pi$ with this property is
$d_\pi := \prod(v_\frakp(\frakN) - v_\frakp(\frakN') + 1)$.

To get our embedding, it suffices to show that $\dim \Theta_\pi \le
d_\pi$.  By \cref{lem32} and \eqref{eq:AL-decomp2},
 $\dim \Theta_\pi \le \dim \pi_\bfk^\Omega \le 2^j \prod_{\frakp | \frakM}
 (v_\frakp(\frakN) - v_\frakp(\frakN')+1)$, where $j$ is the number
 of primes $\frakp | \frakN_2$ such that $\dim \pi_\frakp > 1$.
 Hence it suffices to remove a factor of 2 in this latter bound for each 
 $\frakp | \frakN_2$ such that $\dim \pi_\frakp > 1$ and
 $v_\frakp(\frakN) = v_\frakp(\frakN')$.
Consider such a $\frakp$, and suppose $\phi, \phi' \in \pi_\bfk^\Omega$ are
linearly independent factorizable functions such that 
$\phi_v = \phi'_v$ for $v \ne \frakp$.  
By assuming $\phi, \phi'$ lie in our basis $\Phi$, we see they will only
contribute at most a 1-dimensional space to $\Theta_\pi$
if all Hecke operators $T_\frakn$ act by scalar matrices on $\langle \phi, \phi' \rangle$.  
This is obvious if $\frakp \nmid \frakn$, so it remains to show this for 
each $T_{\frakp^m}$.  But $T_{\frakp^m}$ kills both $\phi$ and $\phi'$
by \cref{prop:hecke-comp1}(iii).
\end{proof}

\begin{rem} \label{rem:theta-ker}
Note the last paragraph of the proof describes the kernel of
the (non-canonical) map $S_\bfk^\new(\calO, \Omega, \omega) \to \Theta_{\bfk+\two}^\new(\calO, \Omega, \omega)$ from \cref{qmf-to-theta}.
(The kernel is canonical up to isomorphism as an $\mathcal H^S$-module.)
  Namely, its dimension is $\sum (2^{s(\pi)} - 1)$ where $\pi$ runs over
representations occurring in $S_\bfk^\new(\calO, \Omega, \omega)$ 
and $s(\pi)$ is the number of primes $\frakp | \frakN_2$ for which 
$\dim \pi_\frakp > 1$.
\end{rem}

We have not shown that the Fourier coefficients of our theta series
for $\frakn$ not prime to $\frakN$ are the Fourier coefficients of a 
corresponding form in $M_{\bfk+\two}(\frakN, \omega)$, however
we can conclude this for the new cuspidal theta series:

\begin{thm} \label{theta-thm2}
We have $\Theta_{\bfk + \two}^\new(\calO, \Omega, \omega) \subset
S_{\bfk + \two}^\new(\frakN, \omega)$.
\end{thm}

\begin{proof} 
Consider $\pi$ in  the decomposition of
$S_\bfk^\new(\calO,\Omega,\omega)$ from \cref{prop:quat-AL-decomp}(i).
For each $\pi$, we can associate an $f_\pi \in S^\new_{\bfk + \two}(\frakN, \omega)$ 
as in the previous proof, which now is unique up to scalars.

Using a basis of $S_\bfk^\new(\calO,\Omega,\omega)$ composed
of bases for each $\pi$ as above, 
we can block diagonalize $\Theta_\omega^\new$ with one block for
$\Theta_\pi$ for each $\pi$.  
Then it suffices to show each $T_{\frakp^m}$ acts on $\pi_\bfk^\Omega$ by $C(\frakp^m, f_\pi)$
for each $\pi$, $\frakp$ and $m$.  This is clear for $\frakp \nmid \frakN$.
For $\frakp | \mathfrak D$,
this follows from \cref{prop:hecke-comp1} and the standard computations
of Hecke operators on local newforms of $\GL_2(F_\frakp)$.  
For $\frakp | \frakM$, this follows from \cref{prop:hecke-comp2}.
\end{proof} 

This proves that the map $S_\bfk^\new(\calO, \Omega, \omega) \to \Theta_{\bfk+\two}^\new(\calO, \Omega, \omega)$ from \cref{qmf-to-theta} is a realization
of the map JL in terms of theta series.

Specializing to the original case of trivial nebentypus, we get the following
solution to the basis problem.  
Note that for any level $\frakN$, we may choose
our quaternion algebra $B$ and special order $\calO$ such that there
is at most one prime dividing $\frakN_2$.

For a space $S$ of modular forms, $S^\psi$ denotes the space of twists by $\psi$.

\begin{cor} \label{cor:basis}
If $\calO$ is unramified quadratic type, i.e., $\frakN_2 = \frako$ then
\[ S_{\bfk + \two}^\new(\frakN, 1) = \Theta_{\bfk + \two}^\new(\calO, 1, 1). \]
If $\frakN_2 = \frakp^{2e}$, then 
\[ S_{\bfk + \two}^\new(\frakN, 1) \subset \Theta_{\bfk + \two}^\new(\calO, 1, 1)
 + \sum_\psi \sum_{j=1}^{e}  \Theta_{\bfk + \two}^\new(\calO(\frakN_1, \frakp^j, \frakM), \Omega_\psi, \psi^{-2})^\psi
+ \sum_\psi S_{\bfk + \two}^\new(\frakN_1 \frakM, \psi^{-2})^\psi, \]
where $\psi$ runs over primitive characters of conductor $\frakp^e$,
$\calO(\frakN_1, \frakp^j, \frakM)$ denotes a special order of level $\frakN_1 \frakp^j
\frakM$ in $B$ locally isomorphic to $\calO$ at all places away from $\frakp$, which is
of unramified quadratic type at $\frakp$ (and thus globally) if $j$ is odd, and 
$\Omega_\psi$ is a character
for this order extending $\psi^{-2}$ as in \cref{sec:qmfs}.
\end{cor}

\begin{proof}
Since we have trivial central character, any cuspidal $\pi$ of conductor $\frakN$
must be discrete series at $\frakp$ when $v_\frakp(\frakN)$ is odd.
Thus, by \cref{JL-thm-new}, $\Theta_{\bfk + \two}^\new(\calO, 1, 1)$
contains all newforms $f \in S_{\bfk + \two}^\new(\frakN, 1)$
except possibly those corresponding to automorphic representations $\pi$
which are not minimal at $\frakp | \frakN_2$.   In particular, the first statement
follows.

Now assume $\frakN_2 = \frakp^{2e}$ and $\pi_\frakp$ is not minimal.  
Then there exists a local character $\psi_\frakp$ such that $c(\psi_\frakp) 
= e$ and $\pi_\frakp \otimes \psi_\frakp^{-1}$
minimal.  Let $\psi^{-1}$ be a finite-order globalization of 
$\psi_\frakp^{-1}$ unramified at all other finite places.
If $\pi_\frakp$ is a (ramified) principal series, then
$c(\psi_\frakp) =  \frac{v_\frakp(\frakN_2)}2$ and 
$\pi_\frakp \otimes \psi_\frakp^{-1}$ is unramified.  Otherwise,
$j = c(\pi_\frakp \otimes \psi_\frakp^{-1}) \ge 1$, and we can replace 
$\calO_\frakp$ by a special order $\calO'_\frakp$ of level $\frakp^j$ 
(of unramified quadratic type if $j$ is odd) to pick up the twist
$\pi \otimes \psi^{-1}$ in a suitable $S^\new_\bfk(\calO', \Omega', \psi^{-2})$.
\end{proof}

The proof of this, without assuming anything on $\frakN_2$, tells us the effectively
weaker result that $S_{\bfk + \two}^\new(\frakN, 1)$ is generated by 
$\Theta_{\bfk + \two}^\new(\calO, 1, 1)$ and twists of Hilbert newforms
of smaller levels at some primes dividing $\frakN_2$.  This is a more precise
version of \cref{thm:intro}, which we think of as a ``weak solution'' to the
basis problem.

Moreover, by varying the quaternion algebra $B$, we see that we will
get every newform in $S_{\bfk + \two}^\new(\frakN, 1)$ via some quaternionic
theta series unless $[F:\Q]$ is odd (so we cannot take $B$ to be unramified
at all finite places) and the associated representation $\pi$ is a principal
series at all ramified places, i.e., $\pi$ is a twist of an unramified representation.
This proves \cref{cor:intro}, our solution to the basis problem for trivial
nebentypus.

The same argument can be applied to modular forms with character,
however now even a weak solution to the basis problem is not as clean
in general: the space
$S_{\bfk + \two}^\new(\frakN, \psi)$ is generated by 
$\Theta_{\bfk + \two}^\new(\calO, \Omega, \psi)$ together with Hilbert
newforms of level $\frakN$ which are minimal ramified principal series
at some primes dividing $\frakN_1 \frakN_2$ and twists of Hilbert newforms of
smaller levels at some primes dividing $\frakN_2$.  Namely, if 
$\frakp | \frakN_1 \frakN_2$ such that $c_\frakp(\psi) = v_\frakp(\frakN)$,
then there is a minimal ramified principal series, $\pi(1,\psi_\frakp)$ in
standard notation, with local conductor $v_\frakp(\frakN)$.  Hence
one will not get all newforms in general spaces $S_{\bfk + \two}^\new(\frakN, \psi)$
by theta series attached to $\calO$ and twists of newforms of smaller level, rather only when $c_\frakp(\psi) < v_\frakp(\frakN)$ for all $\frakp | \frakN_1 \frakN_2$.

However, by choosing $B$ to be ramified at at most one finite prime, and
taking appropriate $\calO$ and $\Omega$, we see that we
can generate $S_{\bfk + \two}^\new(\frakN, \psi)$ by quaternionic theta
series and twist of newforms of smaller level if $[F:\Q]$ is even (where no twists
are needed) or if
$c(\psi) \ne \frakN$.  Note the exception
 $c(\psi) \ne \frakN$ is the analogue of the classical situation with 
 $F=\Q$ and $\psi=1$ where $S_k(1)$ is not generated by
theta series.

%
%

\end{document}